\documentclass[11pt,letterpaper]{amsart}
\usepackage[margin=1.1in]{geometry}
\usepackage[utf8x]{inputenc}
\usepackage[toc,page]{appendix}
\usepackage{microtype}
\usepackage[T1]{fontenc}
\usepackage{enumitem}
\usepackage{tikz-cd}
\usetikzlibrary{positioning,fit,backgrounds}
\usepackage{ae,aecompl}
\usepackage{times}
\usepackage{float}
\usepackage{scalerel}
\usepackage{stmaryrd}
\usepackage{verbatim,amsmath,amsthm,amsfonts,amssymb,latexsym,graphicx,mathtools,color,mathabx,soul}
\usepackage[all]{xy}
\usepackage{graphicx}
\usepackage{caption}
\usepackage{subcaption}
\DeclareMathAlphabet{\mathpzc}{OT1}{pzc}{m}{it}
\usepackage[colorlinks,pagebackref,hypertexnames=false]{hyperref} 
\usepackage{hyperref}
\hypersetup{
  colorlinks = true,
  linkcolor  = black,
  citecolor=blue
} 

\usepackage[alphabetic,backrefs,msc-links]{amsrefs}
\usepackage{accents}

\graphicspath{./images/}
\usepackage{aliascnt}
\numberwithin{equation}{section}

   \newcommand{\HMb}{\overline{\mathit{HM}}}

\newtheorem{theorem}{Theorem}[section]
\newtheorem{prop}[theorem]{Proposition}
\newtheorem{lemma}[theorem]{Lemma}

\newtheorem{cor}[theorem]{Corollary}

\newtheorem*{theorem*}{Theorem}
\newtheorem{theoremint}{Theorem}

\newtheorem{questint}[theoremint]{Question}

\theoremstyle{definition}
\newtheorem{definition}[theorem]{Definition}

\theoremstyle{remark}
\newtheorem{remark}{Remark}[section]
\newtheorem{example}{Example}[section]
\newtheorem*{remarkint}{Remark}

\newcommand{\spin}{\mathfrak{s}}

     \RequirePackage{rotating}                   
    \def\HMt{%
       \setbox0=\hbox{$\widehat{\mathit{HM}}$}
       \setbox1=\hbox{$\mathit{HM}$}
       \dimen0=1.1\ht0
       \advance\dimen0 by 1.17\ht1
       \smash{\mskip2mu\raise\dimen0\rlap{%
          \begin{turn}{180}
              {$\widehat{\phantom{\mathit{HM}}}$}
           \end{turn}} \mskip-2mu    
                \mathit{HM}
    }{\vphantom{\widehat{\mathit{HM}}}}{}}

\newcommand{\HMf}{\widehat{\mathit{HM}}}

\newcommand{\HMr}{\mathit{HM}}

\begin{document}

\title{On integral rigidity in Seiberg--Witten theory}

\author{Francesco Lin}
\address{Department of Mathematics, Columbia University} 
\email{flin@math.columbia.edu}

\author{Mike Miller Eismeier}
\address{Department of Mathematics, University of Vermont} 
\email{Mike.Miller-Eismeier@uvm.edu}

\begin{abstract}
 We introduce a framework to prove integral rigidity results for the Seiberg--Witten invariants of a closed $4$-manifold $X$ containing a non-separating hypersurface $Y$ satisfying suitable (chain-level) Floer theoretic conditions. As a concrete application, we show that if $X$ has the homology of a four-torus, and it contains a non-separating three-torus, then the sum of all Seiberg--Witten invariants of $X$ is determined in purely cohomological terms.
 
 Our results can be interpreted as $(3+1)$-dimensional versions of Donaldson's TQFT approach to the formula of Meng--Taubes, and build upon a subtle interplay between irreducible solutions to the Seiberg--Witten equations on $X$ and reducible ones on $Y$ and its complement. Along the way, we provide a concrete description of the associated graded map (for a suitable filtration) of the map on $\HMb_*$ induced by a negative-definite cobordism between three-manifolds, which might be of independent interest.
\end{abstract}

\maketitle
\setcounter{tocdepth}{2}

\section{introduction}

Seiberg--Witten invariants \cite{WittenOG,Mor} are a fundamental tool in $4$-dimensional topology. Given their versatility, understanding these invariants is an extremely challenging problem that has attracted a considerable amount of attention in the past thirty years. Despite this, we are currently lacking a framework to carry out the computation in general examples. Less ambitiously, it is unclear what kind of constraints such invariants must satisfy. 
\par
A very fruitful line of investigation towards such constraints is that of ``mod $2$ rigidity results'' for the Seiberg--Witten invariants of \textit{spin} manifolds, which assert that the Seiberg--Witten invariants (mod $2$) depend only on some simpler topological information. This started with Morgan--Szab\'o's proof that for a homotopy $K3$ the invariant of the trivial spin$^c$ structure is odd \cite{MorSz}. Ruberman and Strle \cite{RubermanStrle} proved an interesting rigidity result for \textit{homology tori}, 4-manifolds $X$ with the integral homology of a torus. There is a cohomological invariant attached to homology tori, the \textit{determinant}
\begin{equation}\label{determinant}
\det(X)=|\langle x_1\cup x_2\cup x_3\cup x_4, [X]\rangle| \in \mathbb{N},
\end{equation}
where the $x_i$ form an integral basis of $H^1(X)=\mathbb{Z}^4$. Their main result is then the following.
\begin{theorem*}[\cite{RubermanStrle}]
Consider a homology $4$-torus $X$ which is \textbf{spin}. Then the sum of all (degree zero) Seiberg--Witten invariants $\mathfrak{m}(X)$ of $X$ has the same parity as $\det(X)$.
\end{theorem*}
The essential point is that on a spin manifold the Seiberg--Witten equations admit a $\mathrm{Pin}(2)$-symmetry. See \cite{TJLi,Bauer,Baraglia} for further developments in this direction.
\\
\par
The goal of this article is instead to prove some \textit{integral} rigidity results for Seiberg--Witten invariants of a closed 4-manifold $X$ in the presence of a \textit{non-separating} hypersurface $Y$ satisfying suitable conditions. For simplicity, we will not discuss homology orientations and all the results will be stated up to an overall sign. Let us begin by focusing on the special case in which $Y=T^3$ is the three-torus, in which case we have the following purely cohomological formula for the sum of Seiberg--Witten invariants. 

\begin{theoremint}\label{cor:det-squared}
    Suppose $X$ is a homology torus which contains a non-separating 3-torus, and assume that $\sigma(X)=0$ if $\mathrm{det}(X)=0$.\footnote{If $\mathrm{det}(X)\neq 0$, the signature $\sigma(X)$ is automatically zero.} If $X$ admits a spin$^c$ structure restricting to the unique torsion one on $T^3$, then
    \begin{equation*}
         \mathfrak m(X) = \pm \det(X) \cdot\# \left|H^2(T^3)/\mathrm{Im}\left(H^2(X)\rightarrow H^2(T^3)\right)\right|\in\mathbb{Z}.
    \end{equation*}
If not, all its Seiberg--Witten invariants $\mathfrak{m}(X,\spin_X)$ vanish.
\end{theoremint}
Importantly, we do not assume $X$ is spin here, but of course when $X$ is spin there exists a spin$^c$ structure restricting to the torsion one on $T^3$. Furthermore, if $\mathrm{det}(X)$ is odd, $X$ is automatically spin. It is interesting to notice that in this case, the second factor in our formula is independent of the choice of non-separating three-torus.
\begin{remarkint}
While it seems plausible that there exist $4$-manifolds for which the second clause in Theorem \ref{cor:det-squared} applies, the authors do not have concrete examples.
\end{remarkint}
For homology four-tori of the form $X = S^1 \times M$ for some homology three-torus $M$ the work of Meng--Taubes (\cite{MengTaubes}, cf. \cite{RubermanStrle}) implies that
\begin{equation*}
\mathfrak m(X)=\pm \mathrm{det}(X)^2,
\end{equation*}
and one readily checks that when $M$ containes a separating torus this coincides with our formula (see Example \ref{ex:3mfd}). In Example \ref{ex:misdet} we describe a broad class of homology tori for which $\mathfrak m(X) =\pm  \det(X)$ instead.
\begin{remarkint}
Whether or not an exotic $T^4$ exists is an outstanding question in four-dimensional topology.
\end{remarkint}

This first result is a direct consequence of much more general rigidity results for $4$-manifolds $X$ containing a non-separating three-torus $Y=T^3$. For the statement, it will be convenient to work with the complement $W = X \setminus Y$ of the hypersurface; it carries two natural inclusion maps $i,j: Y \to W$, corresponding to the two sides of the hypersurface. In the case of most interest, we will have $b_1(W) = 3$, so the map $$i^*: H^1(W) \to H^1(Y)$$ is a homomorphism of rank 3 free abelian groups. There is a well-defined natural number, the \textit{discriminant of $W$,} 
\begin{equation}\label{discriminant}
\text{disc}(W) = |\det(i^*)| = \#|H^1(Y)/\text{im}(i^*)|,
\end{equation}
which is interpreted as zero if the right hand side is infinite. One could also use the map $j^*$ for the definition, cf. Remark \ref{samecup}. For $\spin$ a spin$^c$ structure on $W$ which is torsion on the ends, write
\begin{equation}\label{dspinW}
d(\spin_W) = \frac 14\left(c_1(\spin_W)^2 - 2 \chi(W) - 3\sigma(W)\right);
\end{equation}
see \cite[Section 28.3]{KM} for the definition of $c_1^2$ in this context.
We also record the quantity
\begin{equation}\label{DW}
D(W) = \# \{\spin_W \mid i^* \spin_W \cong j^* \spin_W \text{ are torsion and } d(\spin_W) = 0\}.
\end{equation}
We have $D(W) \ne 0$ if and only if $X$ supports an almost complex structure $J$ with $c_1(J)|_Y$ torsion, cf. \cite[Appendix $1.4$]{GompfStipsicz}.

\begin{theoremint}\label{thm:torus-rigidity}
    Suppose $X$ is a closed oriented, connected 4-manifold with $b^+(X) \ge 2$. If $T^3\subset X$ is a non-separating three-torus, the sum of all Seiberg--Witten invariants satisfies $$\pm \mathfrak m(X) = \begin{cases} \textup{disc}(W) D(W) &\text{if } b^+(W) = 0 \text{ and } b_1(W) = 3, \\ 0 & \text{otherwise}. \end{cases}$$
\end{theoremint}
\begin{remarkint}
We will refer to a cobordism $W$ with $b^+(W)=0$ as \textit{negative-definite}. For any cobordism $W$, if we consider the cup product restricted to
\begin{equation*}
\hat H(W) = \text{Im}\left(H^2(W,\partial W)/\mathrm{tors}\rightarrow H^2(W )/\mathrm{tors}\right)
\end{equation*}
the resulting pairing is non-degenerate, and $b^+(W) = 0$ is equivalent to the pairing on $\hat H(W)$ being negative-definite.
\end{remarkint}
\begin{remarkint}
With additional care, this result also applies to the case $b^+(X)=1$, cf. the remark after Theorem \ref{thm:weak-calculation} below.
\end{remarkint}

For comparison, if $X$ contains a torus $T^3$ \textit{separating} it in two pieces $X_1,X_2$ with $b^+\geq1$, then the sum of all Seiberg--Witten invariants $\mathfrak{m}(X)$ vanishes, see \cite[Corollary 3.11.2]{KM}; this is a direct consequence of the fact that the reduced Floer homology group $\HMr_{*}(T^3)$ vanishes.


By contrast, the proof of Theorem \ref{thm:torus-rigidity} is more subtle, and is special to $T^3$ in three ways:

\begin{enumerate}
    \item First, $T^3$ has \textit{trivial Thurston norm}, which guarantees that all basic classes on $X$ must restrict to the torsion spin$^c$ structure $\spin_0$ on $T^3$. 
    \item Even though the reduced Floer homology group $\HMr_{*}(T^3,\spin_0)$ vanishes, the Floer homology $$\HMr_{*}(T^3,\spin_0;\Gamma_\eta)\cong \mathbb{R}$$ is \textit{nontrivial} with respect to an appropriate local coefficient system $\Gamma_{\eta}$.
    \item Finally, there are suitable metrics and perturbations so that the Seiberg--Witten equations have \textit{no irreducible solutions}, so that the generators of the Floer chain complex $\hat{C}_*(T^3,\spin_0)$ all arise from unstable reducible critical points; furthermore, the chain complex $\hat{C}_*$ is very simple and concretely understandable.
\end{enumerate}
Given these, the argument roughly proceeds as follows. First of all, a suitable gluing theorem allows to identify $\mathfrak{m}(X)$ as the (super)trace of the cobordism map induced by $W$ (equipped with a suitable local system morphism). Such a map involves counting irreducible solutions to the Seiberg--Witten equations, so is usually no easier to compute than $\mathfrak{m}(X)$ itself.
\par
On the other hand we will see that one can obtain a concrete description (at the chain level) of the associated graded map induced in $\HMb_*$ by $W$. This only counts reducible solutions, and carries interesting information related to the maps
\begin{equation}\label{eq:correspondence}
H^1(Y) \xleftarrow{i^*} H^1(W)\xrightarrow{j^*} H^1(Y)
\end{equation}
induced by inclusion. Because in our situation the Floer chain complex $\hat{C}_*$ is closely related to $\bar{C}_*$, this will allow to reconstruct the cobordism map we are actually interested in.
\\
\par
In particular, the fundamental mechanism behind the rigidity result in Theorem \ref{thm:torus-rigidity} is a subtle interplay between irreducible solutions to the Seiberg--Witten equations on $X$ and the reducible ones on the hypersurface $T^3$ and its complement $W$. 
\par
Our strategy can be in fact interpreted as a $(3+1)$-dimensional version of Donaldson's \cite{Don} $(2+1)$-dimensional TQFT approach to the formula of Meng--Taubes \cite{MengTaubes}. There he associates to surfaces $\Sigma_0,\Sigma_1$ the cohomology of their symmetric products (the latter being the moduli space of solutions to the vortex equations), and to a cobordism $W$ between them the cohomology of the moduli space of solutions to the Seiberg--Witten equations. From the computation of the cohomology of symmetric products and naturality arguments, the Alexander polynomial arises as the trace for self-cobordisms in the simpler TQFT associating to a surface $\Sigma$ the group $\Lambda^* H^1(\Sigma;\mathbb{Z})$ and to a cobordism the map induced by the correspondence as in (\ref{eq:correspondence}).
\par
The $(3+1)$-dimensional picture is significantly more subtle: both in that the study of the Seiberg--Witten equations on three-manifolds is much richer than the case of surfaces, and that the relation with Donaldson's TQFT only holds in the negative definite situation. In this case, our results show that in fact Donaldson's TQFT determines for a suitable filtration the \textit{associated graded} of the map in $\HMb_*$ induced by $W$. While this is all we need for our trace computations, the higher order terms are expected to have a more subtle relationship to the topology of Dirac operators on $W$.
\\
\par
Given the key properties used in the outline above, it is natural to ask for generalizations of the result to other types of hypersurfaces. If a $4$-manifold $X$ (which we assume to have $b^+\geq 1$ so that one can define Seiberg--Witten invariants) admits a hypersurface $Y$ with negative definite complement, then $b_1(Y)\geq 1$, so we will assume it throughout the paper. From a Floer-theoretic perspective, the most important and interesting assumption is $(3)$. This leads us to introduce the notion of \textit{$RSF$-space} in Definition \ref{maindef}, where $R$ and $SF$ mean \textit{reducible} and \textit{strictly filtered} respectively. These are torsion spin$^c$ three-manifolds $(Y,\spin)$ with $b_1\geq 1$ satisfying, for suitable metric and perturbations, a condition regarding the complexes $\hat{C}_*(Y,\spin)$ and $\bar{C}_*(Y,\spin)$ and the map relating them. Being a chain-level condition, it is quite subtle to check in practice. Other than the three-torus, it also holds in other interesting examples with $b_1\geq 1$:
\begin{itemize}
    \item the flat three-manifolds with $b_1=1$ for any torsion spin$^c$ structure, following \cite[Section $37.4$]{KM};
    \item the mapping tori of a finite order mapping class $\varphi$ of a surface $\Sigma$ of genus $g\geq 2$ with $\Sigma/\varphi=\mathbb{P}^1$, equipped with a self-conjugate spin$^c$ structure, following \cite{LinTheta};
    \item the product of a circle with a surface of genus $2$ or $3$ for the unique torsion spin$^c$ structure, following \cite{spectrally-large}.
\end{itemize}
In this more general situation, we have the following integral rigidity result for the sum of the invariants corresponding to spin$^c$ structures on $X$ restricting to the given one on $Y$.

\begin{theoremint}\label{thm:weak-calculation}
Suppose $X$ is a closed, oriented, connected $4$-manifold with $b^+(X) \ge 1$. Suppose further that $ Y \subset X$ is a non-separating hypersurface equipped with a torsion spin$^c$ structure with $(Y,\spin_Y)$ an $RSF$-space. Then for each spin$^c$ structure $\spin_W$ on $W$ with $i^* \spin_W \cong j^* \spin_W = \spin_Y$  the Seiberg--Witten invariants of $X$ satisfy $$ \sum_{\substack{\spin_X\lvert_{W} = \spin_W}} \mathfrak m(X,\spin_X) = \begin{cases} c(W,Y, \spin_Y) &\text{if } b^+(W) = 0 \text{ and } b_1(W) = b_1(Y)\text{ and } d(\spin_W) = 0 \\ 0 & \text{otherwise}. \end{cases}$$ where $c(W,Y, \spin_Y)\in\mathbb{Z}$ is a quantity that depends only on the the correspondence (\ref{eq:correspondence}) and the spin$^c$ structure $\spin_Y$ on $Y$.
\end{theoremint}
\begin{remarkint}
When $b^+=1$, one in general has two Seiberg--Witten invariants depending on the side of the wall, but it turns out that in the setup of the theorem they coincide, cf. Lemma \ref{nowallcrossing} below. One could get a statement involving all Seiberg--Witten invariants like Theorem \ref{cor:det-squared} for homology $S^2 \times T^2$ containing nonseparating flat submanifolds with $b_1 = 1$, if one were willing to be somewhat more careful about the chamber structure in Lemma \ref{lemma:Thurston} when $b^+(X) = 1$ (cf. \cite[Section $27.5$]{KM}).
\end{remarkint}

The quantity $c(W,Y,\spin_Y)$ can be computed explicitly provided one has a concrete understanding of the Floer theory of $(Y,\spin_Y)$; this can be done for the examples pointed out above, see Section \ref{c-comp}.
\\
\par
It is interesting to compare the notion of $RSF$-space, which assumes $b_1(Y) > 0$, to related conditions for $b_1(Y) = 0$. Among manifolds with $b_1(Y) = 0$, those which support a metric and with no irreducible solutions are called \emph{minimal $L$-spaces}. The relevant result is then the following: if $X$ has non-vanishing Seiberg--Witten invariants and it contains a minimal $L$-space as a separating hypersurface then one of the two sides must have $b^+ = 0$. Of course, this result still applies to a much broader class of rational homology spheres, the \emph{$L$-spaces}, which may admit irreducible solutions but still have trivial reduced Floer homology. It would be interesting if our $RSF$-spaces were akin to the minimal $L$-spaces: members of a useful family of manifolds defined by a weaker, homology-level condition. 

\begin{questint}
Is there a more general notion than $RSF$-space, defined by a homology-level condition, which guarantees rigidity results along the lines of Theorem \ref{thm:weak-calculation}?
\end{questint}

To conclude, the Ozsv\'ath--Szab\'o mixed invariants \cite{OzsSza4man} are conjectured to be equal to the Seiberg--Witten invariants. Even though the three-manifold invariants in Heegaard and monopole Floer homology are isomorphic (\cite{CGH,KLT} and subsequent papers), the theories contain very different chain-level information. It would therefore be interesting to reprove some of our results for in that context, starting for example from the following.

\begin{questint}
Can Theorem \ref{thm:torus-rigidity} be proved for the Ozsv\'ath--Szab\'o mixed invariants? 
\end{questint}

A proof of this might suggests some generalization of our results in the spirit of Question D. Notice that the Ozsv\'ath--Szab\'o invariants are currently 
only known to be well-defined with values in $\mathbb{Z}/2$, but a positive answer to Question E mod $2$ would be interesting nonetheless.
\vspace{0.3cm}

\textbf{Organization of the paper.} In Section \ref{gluing}, we state a self-gluing formula (well-known to experts) for the Seiberg--Witten invariants of $X$ in terms of the trace of the map induced on the Floer homology of a separating hypersurface $Y$ by its complement. In Section \ref{negative}, we give a chain-level description of the associated graded map in $\HMb_*$ induced by a negative definite cobordism. This is a key computation behind our results, and might be of independent interest. In Section \ref{RSF}, we introduce $RSF$-spaces and prove the general integral rigidity result, Theorem \ref{thm:weak-calculation}. In Section \ref{c-comp} we discuss some concrete computations involving it. In Section \ref{torus}, we specialize to the case of $Y=T^3$ and deduce Theorem \ref{thm:torus-rigidity} and Theorem \ref{cor:det-squared} above.

\section{Gluing formulas in the non-separating case}\label{gluing}

Suppose $X$ is a closed oriented 4-manifold with $b^+(X) \ge 2$, and that $Y \subset X$ is a non-separating hypersurface; write $W: Y \to Y$ for the complement, considered as a cobordism. Write $i,j: Y \hookrightarrow W$ for the inclusion of the two boundary components. In order to relate the Seiberg--Witten invariants on a closed 4-manifold $X$ to the induced map of the complementary cobordism $W$, we need a self-gluing formula, which we now state. We follow quite closely the discussion of formal properties in \cite[Chapter $3$]{KM}, to which we refer for more details.

Consider the Seiberg--Witten generating function
\begin{equation*}
\mathfrak m(X, \spin_X, h) = \mathfrak{m}(X,\spin_X)e^{\langle c_1(\spin_X),h\rangle}\text{ for } h\in H_2(X;\mathbb{R}).
\end{equation*}
Note that $\mathfrak{m}(X,\spin_X)$ is nonzero only when the formal dimension $d(\spin_X)$ is zero. 

If $\spin_W$ is a spin$^c$ structure on $W$ so that $i^* \spin_W \cong j^* \spin_W$, it is isomorphic to the restriction of many spin$^c$ structures $\spin_X$ on $X$. It will be useful to consider sums such as $$\mathfrak m(X, \spin_W, h) \overset{\text{def}}{=} \sum_{ \spin_X\lvert_W \cong \spin_W} \mathfrak m(X, \spin_X, h).$$
Moving to gluing formulas, let us consider a possibly non-torsion spin$^c$ structure $\spin_Y$ on $Y$. Suppose now that $h=[\bar \nu]$ for a $2$-cycle $\bar \nu \in C_2(X; \Bbb R)$, and consider $\eta=\bar \nu\cap Y$; for generic choice of $\bar{\nu}$, this is a $1$-cycle in $Y$. Cutting open, we obtain a $2$-chain $\nu$ in $W$ with $\partial\nu=-\eta\sqcup \eta$.
Assuming $[\eta]\neq 0\in H_1(Y;\mathbb{R})$, the completed bar Floer homology group with local coefficients $\HMb_{\bullet}(Y,\spin_Y;\Gamma_\eta)$ vanishes, so that the natural map $$\HMt_\bullet(Y, \spin_Y; \Gamma_\eta) \to \HMf_\bullet(Y, \spin_Y; \Gamma_\eta)$$ is an isomorphism and hence both groups are identified with the the \textit{reduced Floer homology group} $$\mathit{HM}_\bullet(Y, \spin_Y; \Gamma_\eta) = \text{Im}\left(\HMt_\bullet(Y, \spin_Y; \Gamma_\eta) \to \HMf_\bullet(Y, \spin_Y; \Gamma_\eta)\right).$$
The chain $\nu$ in $W$ gives rise to the morphism $\Gamma_\nu$ of local systems $\Gamma_\eta\rightarrow \Gamma_\eta$, which determines the map $\mathit{HM}_\bullet(W, \spin_W; \Gamma_\nu)$. Recalling that the Floer homology groups have a canonical $\mathbb{Z}/2\mathbb{Z}$-grading, we can state the self-gluing equation as follows.
\begin{prop}\label{prop:gluing}
With notation as above, assuming $[\eta]\neq 0$, we have an equality
\begin{equation*}
\mathfrak{m}(X, \spin_W,[\bar \nu]) = \textup{Tr}\left(\mathit{HM}_{\bullet}(W,\spin_W;\Gamma_{\nu})\right) = \textup{Tr}\left(\HMf_{\bullet}(W,\spin_W;\Gamma_{\nu})\right)
\end{equation*}
where $\textup{Tr}$ denotes the supertrace (or `alternating trace').
\end{prop}

One should compare this with the case of a separating hypersurface treated in the \cite[Proposition 3.9.3]{KM}; in particular, in the same way that result allows the pieces to have $b^+=0$, our proposition also holds when $W$ is negative definite. The proof of this self-gluing result very closely follows the separating case discussed in \cite[Chapter $32$]{KM}, and involves comparing the isomorphism between $\mathit{HM}_{\bullet}(Y,\spin_Y;\Gamma_\eta)$ and some corresponding Floer homology groups with non-balanced perturbations (which only involves irreducible solutions) introduced in \cite[Chapter $30$]{KM}. Regarding the analytical aspects of the proof, even though \cite{KM} only deals with gluing results for moduli spaces under neck-stretching along separating hypersurfaces (see Chapters $26$ and $31$), the proofs of the relevant statements carry over to the non-separating setup without significant difficulties.
\begin{remark}
We will be only considering spin$^c$ structures which are torsion in the rest of the paper, so we can equivalently work with the uncompleted version of the groups $\mathit{HM}_*(Y, \spin_Y; \Gamma_\eta)$.
\end{remark}
A similar gluing formula holds in the case of a $4$-manifold $X$ with $b^+=1$ containing a non-separating hypersurface $Y$. In this setup one needs to be careful because there are two (usually distinct) invariants depending on the side of the wall one is looking at. On the other hand, in the situation of torsion spin$^c$ structures which is of interest to us we have the following.

\begin{lemma}\label{nowallcrossing}
Suppose $X$ has $b^+=1$ and contains a non-separating hypersurface $Y$. If $\spin_X$ is a spin$^c$ structure such that $\spin_X\lvert_Y$ is torsion, the Seiberg--Witten invariant $\mathfrak m(X, \spin_X)$ is independent of the side of the wall.
\end{lemma}

\begin{proof}
This follows from the general wall-crossing formula \cite[Corollary 1.3]{LiLiu}. The dimension of the Seiberg--Witten moduli space has the same parity as $1 - b_1 + b^+ \equiv b_1$, so one can assume $b_1(X)$ is even. Then, for any basis $\{y_i\}$ of $H^1(X;\mathbb{Z})$ the authors show for the spin$^c$ structure $\spin_X$ the difference between the two invariants is the Pfaffian of the $(b_1\times b_1)$-dimensional skew-symmetric matrix with $(i,j)$-th entry
\begin{equation}\label{pfaff}
\langle y_i\cup y_j\cup c_1(\spin_X),[X]\rangle/2.
\end{equation}
Now consider the Mayer-Vietoris type exact sequence
\begin{equation*}
0\rightarrow \mathbb{Z}\rightarrow H^1(X)\rightarrow H^1(W)\stackrel{i^*-j^*}{\longrightarrow}H^1(Y).
\end{equation*}
Because the groups involved are free abelian, we can choose a basis of $H^1(X)$ for which $y_1$ corresponds to $\mathbb{Z}$ and $y_k$ corresponds to the kernel of $i^*-j^*$ for $k\geq 2$. Then $Y$ is Poincar\'e dual to $y_1$, and we can represent $y_k$ by smooth $3$-manifolds $Y_k$ transverse to $Y$. When one of $i,j$ is one, the quantity in (\ref{pfaff}) can be interpreted as the evaluation of $c_1(\spin_X)/2$ on $Y\cap Y_k$, which is zero because the restriction of $\spin_X$ to $Y$ is torsion. Hence the matrix has first row and column zero, so that its Pfaffian is zero.
\end{proof}

Hence for a spin$^c$ structure restricting to a torsion one on $Y$ we can refer unambiguously to its Seiberg--Witten invariant $\mathfrak{m}(X,\spin_X)$ even when $b^+(X) = 1$; because this will be the case concerning us in the paper, we will state the self-gluing formula in this context. As before we will consider, for a spin$^c$ structure $\spin_W$ on $W$ restricting to $\spin$ on both ends, the quantity $\mathfrak{m}(X,\spin_W,[\nu])$. We then have the following.
\begin{prop}
Proposition \ref{prop:gluing} continues to hold when $b^+(X) = 1$ and $\spin_W$ restricts to a torsion spin$^c$ structure on $Y$.
\end{prop}

This follows as in the case of $b^+\geq 2$; the corresponding result for separating hypersurfaces is \cite[Proposition $27.5.1$]{KM}.

\begin{remark}
The analogue of Proposition \ref{prop:gluing} for the four-manifold invariants in Heegaard Floer homology (with values in $\mathbb{Z}/2$) can be found in \cite[Theorem 1.5(2)]{Zemke}.
\end{remark}

\section{The map induced by a negative definite cobordism}\label{negative}
Given a spin$^c$ cobordism $(W,\spin_W)$ between torsion spin$^c$ three-manifolds $(Y_{\pm},\spin_\pm$), in this section we are concerned with the induced map
\begin{equation*}
        \HMb_{*}(W,\spin_W):\HMb_{*}(Y_-,\spin_-)\rightarrow \HMb_{*}(Y_+,\spin_+).
\end{equation*}
When $b^+(W)>0$, this vanishes \cite[Proposition 27.2.4]{KM}, so we will always assume that $W$ is negative definite. More generally, we will be interested in the case of Floer homology groups equipped with local systems $\Gamma_{\eta_\pm}$.
\par
The simplest case, in which both $Y_\pm$ are rational homology spheres and $b_1(W)=0$, is discussed in \cite[Proposition 39.1.2]{KM}. In this situation, we have the natural identification
\begin{equation*}
\HMb_*(Y_\pm,\spin_\pm)\cong \mathbb{Z}[U,U^{-1}]
\end{equation*}
of absolutely graded $\mathbb{Z}[U]$-modules (up to an overall grading shift, where $U$ has grading $-2$), and the map $\HMb_*(W,\spin_W)$ is an isomorphism. This computation is the key result underlying the topological applications of the Fr\o yshov invariant (see also \cite{FroQHS3}).
\begin{remark}
In the Heegaard Floer setting, the analogue of this computation can be found in \cite{OzSzabs}. Some special cases of the map induced by a negative definite cobordism can be found in \cite{BehGol,LevRub1,LevRub2} with applications to generalized correction terms. For the monopole counterparts of the latter, see \cite{Kru}.
\end{remark}
We will consider the general case in which $b_1(Y_\pm)$ and $b_1(W)$ are arbitrary. Associated to each of these manifolds are the tori
\begin{align*}
\mathbb T_{Y_\pm}&=H^1(Y_\pm; i\mathbb R)/H^1(Y_\pm; 2\pi i \mathbb Z)\\
\mathbb T_{W}&=H^1(W; i\mathbb R)/H^1(W; 2\pi i \mathbb Z).
\end{align*}
By Hodge theory we can identify the former (after choosing a basepoint) as the space of gauge equivalence classes of flat spin$^c$ connections on $(Y_\pm,\spin_\pm)$; we will discuss an analogous interpretation of the latter later. The diagram of maps $$\mathbb T_{Y_-} \xleftarrow{i_-^*} \mathbb T_W \xrightarrow{i_+^*} \mathbb T_{Y_+}$$ induced by the inclusion maps $i_{\pm}$ will play a crucial role in our calculation.
\par
Finally, our discussion will also apply to more general local systems and morphisms between them, but for simplicity we will focus on the local systems of the form $\Gamma_\eta$ for a real 1-cycle $\eta \in C_1(Y; \Bbb R)$ introduced in \cite[Section $3.7$]{KM}; these suffice for the cases of interest to us.

\subsection{Review of the three-manifold case} Let $Y$ be a three-manifold. As in \cite[Section $35.1$]{KM}, given a metric on $\mathbb T_Y$ and a Morse function $f: \mathbb T_Y \to \mathbb R$, one may choose an appropriate perturbation of the Seiberg--Witten equations so that all reducible critical points and trajectories are cut out transversely. In this situation, Kronheimer and Mrowka give a description of $\bar{C}_*(Y,\spin)$ in terms of `coupled Morse theory'. The structure of this complex is determined in \cite[Section 33.3]{KM}, which we record.

\begin{lemma}\label{lemma:Morse-iso}
In the situation above, one has an isomorphism of relatively $\Bbb Z$-graded complexes over $\mathbb Z[U]$ $$\bar{C}_*(Y, \mathfrak s; \Gamma_\eta) \cong C_*(\mathbb T_Y, f; \Gamma_\eta)[U, U^{-1}].$$ With respect to this isomorphism, the differential $\overline \partial$ is sent to $$\partial_f + \partial_3 U^{-1}+\partial_5 U^{-2}+\dots$$ with $\partial_f$ the Morse differential on $C_*(\mathbb T_Y, f; \Gamma_\eta)$.
\end{lemma}

In \cite[Section $34.2$]{KM} the authors further show that one can deform the family of operators (in a suitable space of operators on a Hilbert space) so that the for the resulting coupled Morse complex the terms $\partial_{2i+1}$ with $i\geq 2$ are zero. On the other hand, in the present paper we will work with small perturbations of the family (of geometric nature) so in general one should expect the higher terms to be non-vanishing. 
\par
It is worth being more precise about the identifications in this isomorphism. For each critical point $q \in \text{Crit}(f)$, choose a labeling of the eigenvalues of the Dirac operator $D_q$ by $$\cdots < \lambda_{-1}(q) < \lambda_0(q) < \lambda_1(q) < \cdots$$ For some $i$ we have $\lambda_{i-1}(q) < 0 < \lambda_i(q)$; we record $d(q) = i$ as the shift of this labeling from the natural labeling. Finally, we demand that these labelings respect the spectral flow, in the sense that $$\text{sf}(q_1, q_2) = d(q_2) - d(q_1);$$
this is possible because $c_1(\spin)$ torsion implies that there is no spectral flow around loops. Writing $$\tilde u_2: \bar C(Y, \mathfrak s; \Gamma_\eta) \to \bar C(Y, \mathfrak s; \Gamma_\eta)$$ for the chain map whose induced map on homology is the $U$-action, \cite[Lemma 33.3.9]{KM} guarantees $\tilde u_2$ is an isomorphism. The inverse of our identification sends $qU^i$ to $\tilde u_2^i(q, \lambda_0(q))$, which is equal to the sum of $(q, \lambda_{-i}(q))$ and critical points of strictly smaller Morse index.
\\
\par
For the purposes of the paper, we will \textit{not} be interested in the actual $U$-action on $\bar{C}_*(Y, \mathfrak s; \Gamma_\eta)$, but only on the natural filtration by powers of $U$, with
\begin{equation}\label{filtration}
\mathcal F_k \bar{C}_* \cong C_*(\mathbb T_Y, f)\otimes U^{\lfloor k/2\rfloor} \mathbb{Z}[U^{-1}],
\end{equation}
i.e. the terms with $U$-power $\leq k/2$. While this is not a filtration of $U$-modules, multiplication by $U$ sends each $\mathcal F_k$ isomorphically onto $\mathcal F_{k+2}$, and does induce the structure of a $\Bbb Z[U, U^{-1}]$-module on the pages of the associated spectral sequence. The $E^1$ page of the associated spectral sequence is precisely the Laurent polynomial ring over the Morse complex $C_*(\mathbb T_Y, f; \Gamma_\eta)[U, U^{-1}]$. 

\begin{remark}
The $U$-filtration is seen to be essentially equivalent to the filtration by \textit{Morse index}, $$\mathcal F^{Morse}_k = C_{* \le k}(\mathbb T_Y, f) \otimes \mathbb Z[U, U^{-1}];$$ we have $\mathcal F_k \bar C_d = \mathcal F^{Morse}_{d+k} \bar C_d.$ 
\end{remark}

In \cite[Section 35]{KM}, the authors determine the differential $\partial_3$ on the $E^3$ page of this spectral sequence in terms of the triple cup product
    \begin{align*}
    \cup_Y: \Lambda^3 H^1(Y;\mathbb{Z})&\rightarrow \mathbb{Z}\\
    \alpha\wedge\beta\wedge\gamma&\rightarrow \langle\alpha\cup\beta\cup\gamma,[Y]\rangle
\end{align*}
    and show that the higher differentials in this spectral sequence vanish over $\Bbb Q$. As a consequence, they establish a canonical isomorphism
\begin{equation}\label{bariscup}
\text{gr } \HMb_*(Y,\spin;\mathbb{Q})\cong {HC}_*(Y;\mathbb{Q})
\end{equation} where the former is the associated graded group of the $U$-filtration on homology, and the latter is the \textit{cup homology} \cite{Mar}. The cup homology is the homology of the chain complex whose underlying module is
\begin{equation*}
CC_*(Y)=\Lambda^* H^1(Y;\mathbb{Z})\otimes \mathbb{Z}[U^{-1},U]
\end{equation*}
and whose differential is given by
\begin{equation}\label{d3}
d^3(\omega\otimes U^n)= \iota_{\cup_Y^3}\omega\otimes U^{n-1},
\end{equation}
where $\iota_{\cup_Y^3}$ is the contraction with the triple cup product $\cup_Y^3$ sending
$\alpha_1\wedge\dots\wedge\alpha_k$ to
\begin{equation*}
\sum_{i_1<i_2<i_3} (-1)^{i_1+i_2+i_3}\langle \alpha_{i_1}\cup \alpha_{i_2}\cup \alpha_{i_3},[Y]\rangle\cdot \alpha_1 \wedge\dots  \wedge\hat{\alpha}_{i_1} \wedge\dots \wedge\hat{\alpha}_{i_2} \wedge\dots \wedge\hat{\alpha}_{i_3}\wedge\dots\wedge\alpha_k.
\end{equation*}

\begin{remark}
    In \cite{HMb-Y}, the authors give a combinatorial chain-level model for $\overline{CM}_*(Y, \spin; \Gamma_\eta)$ over the integers, but it remains open whether or not the higher differentials in the associated spectral sequence vanish when taken with integer coefficients. We will not need so refined of an analysis in this paper.
\end{remark}

\subsection{The formula}
Whereas \cite[Section 35]{KM} determined the homology groups $\HMb$, we will now identify the cobordism maps (up to filtration).

In the present situation, we have a correspondence $$(\mathbb T_{Y_-}, f_-) \xleftarrow{i_-^*} \mathbb T_W \xrightarrow{i_+^*} (\mathbb T_{Y_+}, f_+),$$ with $f_\pm$ Morse functions on the tori $\mathbb T_{Y_\pm}$, and the maps $i_\pm^*$ affine. So long as the maps $i_\pm^*$ are generic with respect to the Morse functions $f_\pm$ --- which can achieved by a translation --- there is a well-defined induced map on Morse complexes, given by counting intersection points of $\mathbb T_W$ with $U_a \times S_b$ in $\mathbb T_{Y_-} \times \mathbb T_{Y_+}$ \cite[Section 2.8]{KM}. This map will be denoted $m_W$. 

In the case that $Y_-$ and $Y_+$ are equipped with real 1-cycles $\eta_\pm \in C_1(Y_\pm; \Bbb R)$, a real 2-chain $\nu \in C_2(W; \Bbb R)$ whose boundary is $\partial \nu = i_+ \eta_+ - i_- \eta_-$ provides an isomorphism of local systems $$\Gamma_\nu: i_-^* \Gamma_{\eta_-} \cong i_+^* \Gamma_{\eta_+}.$$ In this situation, we have an induced map on Morse complexes $$m_{W,\nu}: C_*(Y_-, f_-; \Gamma_{\eta_-}) \to C_*(Y_+, f_+; \Gamma_{\eta_+}),$$ also described in \cite[Section 2.8]{KM}, and defined in a similar fashion. 

\begin{theorem}\label{thm:constant-term}
Suppose $(W,\spin_W): (Y_-, \spin_-) \to (Y_+, \spin_+)$ is a cobordism with $b^+(W) = 0$ and $\spin_\pm$ torsion. With respect to the isomorphisms of Lemma \ref{lemma:Morse-iso}, the map $\bar{m}_*$ inducing $\HMb_*(W, \spin_W; \Gamma_\nu)$ is filtered, and takes the form $$\bar{m}_* = m_{W, \nu} U^d + m_2 U^{d-1} + m_4 U^{d-2} + \cdots$$ for an appropriate integer $d$ and appropriate maps $m_{2i}$ for $i > 0$.
\end{theorem}

That is, if we consider the filtration $\mathcal{F}_k\bar{C}_*$ in (\ref{filtration}),
the map $\overline{m}_*(W, \spin_W; \Gamma_\nu)$ is a filtered map with associated graded map equal to the induced map on the associated Morse complex. As a direct consequence, we obtain the following computation, which might be of independent interest.
\begin{cor}
Under the identification with cup homology in (\ref{bariscup}), the associated graded map of $\HMb_*(W, \spin_W;\mathbb{Q})$ is the map in homology induced by the map $(m_W)_*$ on the cup chain complex.   
\end{cor}
\begin{remark}\label{samecup}
Notice that the latter is actually a chain map on the cup chain complex; this follows because $i_+^*(\cup_{Y_+})$ is the same as $i_-^*(\cup_{Y_-})$ as maps from $\Lambda^3H^1(W)$ to $\mathbb{Z}$. More generally, for a manifold $W$ with possibly disconnected boundary $\partial W$, let
\begin{equation*}
r:H^1(W)\rightarrow H^1(\partial W)   
\end{equation*}
denote the restriction map; we denote the induced map on the exterior algebras by $r$ as well. Then the pull back of the triple cup product of $\partial W$ under $r$ vanishes. This can be seen geometrically as follows: the map $r$ corresponds via Poincar\'e duality to the boundary map
\begin{equation*}
\partial: H_3(W,\partial W)\rightarrow H_2(\partial W).
\end{equation*}
Given three elements $\alpha_i$  in $H^1(\partial W)$, represent their duals by surfaces $S_i$ in $\partial W$.
If the $\alpha_i$ are in the image of $r$, then $S_i$ is the boundary of three-manifolds $T_i$ in $W$. After moving everything to a transverse position, $S_1\cap S_2\cap S_3$ (which is dual to $\alpha_1\wedge \alpha_2\wedge\alpha_3$) is a $0$-manifold which is the boundary of $T_1\cap T_2\cap T_3$, hence is zero in homology.
\end{remark}

\begin{remark}
A version of Theorem \ref{thm:constant-term} with more general local systems proves a weaker version of \cite[Conjecture 7.3]{Kru}, and gives a gauge-theoretic proof of \cite[Theorem 7.4]{Kru}.
\end{remark}

These results deserve some further comment. First of all, the isomorphism in Lemma \ref{lemma:Morse-iso} is only defined up to multiplication by powers of $U$. This can be pinned down by using the absolute $\Bbb Q$-grading of \cite[Section 28.3]{KM}, according to which the degree of $\HMb_*(W;\spin_W)$ is given by
\begin{equation}\label{grading}
\frac{1}{4}c_1^2(\spin_W)-\iota(W)-\frac{1}{4}\sigma(W) = d(\spin_W) - \frac 12 \big(b_1(Y_-) - b_1(Y_+)\big)
\end{equation}
where $d(\spin_W)$ is the expression (\ref{dspinW}) and
\begin{equation*}
\iota(W)=\frac{1}{2}\big(\chi(W)+\sigma(W)+b_1(Y_-)-b_1(Y_+)\big).
\end{equation*}
As $U$ has grading $-2$, and the degree of $m_{W, \nu}$ is $b_1(W) - b_1(Y_-)$, we may thus compute the integer $d$ from Theorem \ref{thm:constant-term} as $$d = -\frac 12 d(\spin_W) - \frac 14 \left(2b_1(W) - b_1(Y_-) - b_1(Y_+)\right).$$ In the case of most interest to us in the current paper, we have $Y_- = Y_+$ and $b_1(W) = b_1(Y_-)$ so $d = -\frac 12 d(\spin_W).$

Secondly, even though it will not be needed for this paper, for some applications one would like to compute the map itself rather than the associated graded map. One might hope to compute the higher terms up to (filtered) homotopy, and expect them to be related to the kernel of the Dirac operators on $W$ coupled to the spectral decomposition associated to $D_{Y_\pm}$. An important technical complication when trying to write down an explicit formula in homology is that while there is a canonical identification of the associated graded module of $\HMb_*$ \eqref{bariscup}, there is no known canonical lift to $\HMb_*$ itself.

\subsection{Proof of Theorem \ref{thm:constant-term}} 
Recall from \cite[Chapter $25$]{KM} that the map induced by a cobordism $(W,\spin_W)$ on $\HMb_*$ (possibly with local coefficient systems) is obtained by counting solely reducible solutions to the Seiberg--Witten equations in the blow-up on the cobordism $W^*$ with cylindrical ends attached. The unperturbed equations for a reducible configuration $(A,0,\varphi)$ in the blow-up are
\begin{equation*}
    \begin{cases}
    D_A^+\varphi=0\\
    F_{A^t}^+=0,
    \end{cases}
\end{equation*}
which one needs to study with the right asymptotic growth conditions for $\varphi$ on the cylindrical ends; different asymptotic conditions change the index of the Dirac operator $D_A^+$. We are counting solutions to this equation of index zero. We may break this count up into a sum$$\bar{m}_*=m_0 + m_2 + \cdots$$ as follows. The map $m_{2i}$, by definition, counts solutions $(A, 0, \phi)$ to a perturbed version of these equations for which $A$ lies in a $2i$-dimensional moduli space (or equivalently, is asymptotic to flat connections whose Morse index has difference $2i-(b_1(W)-b_1(Y_+))$, see the formula on p. 47 of \cite[Section 2.8]{KM}) and the asymptotic conditions are chosen so that $\text{ind}_{\mathbb C}(D_A^+) = 1-i$; after projectivizing one obtains a finite number of points. The highest-filtration term in $m_0$ then corresponds to the count of index-zero solutions to a perturbed version of the second equation and asymptotic conditions for which $\text{ind}_{\mathbb{C}}(D_A^+) = 1$. 
\par
The second equation is simply the ASD equation, and its solutions up to gauge with $L^2$-curvature form the $b_1(W)$-dimensional torus
\begin{equation*}
\mathbb{T}_W=H^1(W;i\mathbb{R})/H^1(W;2\pi i\mathbb{Z}).
\end{equation*}
Here, following \cite[Proposition 4.9]{APS1}, we can interpret the numerator as a space $L^2_{\text{ext}}$ of extended $L^2$-harmonic $1$-forms $\alpha$, i.e. the harmonic $1$-forms that exponentially decay on each end to a time-independent harmonic $1$-form $\alpha_{\pm}$. Sending each such form to the limiting value is the Hodge theoretic incarnation of the maps
\begin{equation*}
i^*_{\pm}:H^1(W)\rightarrow H^1(Y_\pm).
\end{equation*}
In the isomorphism of Lemma \ref{lemma:Morse-iso}, we use a perturbation of the form $f_\pm \circ p_\pm$ where $f_\pm$ are our chosen Morse functions and  
\begin{equation*}
p_\pm:\mathcal{B}(Y_\pm,\mathfrak{s})\rightarrow \mathbb{T}_{Y_\pm}
\end{equation*}
is the retraction obtained by mapping $(B,\Psi)$ via the $L^2$-orthogonal projection of the $1$-form $B-B_0$ to its harmonic part $(B-B_0)^{\mathrm{harm}}$ given by the Hodge decomposition\footnote{We have fixed a reference flat spin$^c$ connection $B_0$ here.}. We use these to perturb the Seiberg--Witten equations as in \cite[Chapter $24$]{KM}, using on $[0,\infty)\times Y_+$ a smooth cutoff function $$\beta: [0, \infty)\to [0,1]$$ which is nondecreasing, equal to zero near zero and equal to $1$ for $t \ge 1$, and similarly on the other end.
\par
We then consider solutions $(A,\varphi)$ to a perturbed Seiberg--Witten equation for which the gauge equivalence classes $A\lvert_{(-\infty,0]\times Y_-}$ and $A\lvert_{[0,+\infty)\times Y_+}$ converge to limits $q_\pm$, reducible solutions on $Y_{\pm}$ which are critical points of $f_{\pm}$.
We demand $\varphi$ satisfies some asymptotic growth conditions at $\pm\infty$ in terms of the eigenvalues and eigenspinors of the limit Dirac operators $D_{q_{\pm}}$ (which can also be rephrased in terms of convergence in the $\tau$-model).

Again let us focus on the perturbed ASD equation. The torus $\mathbb T_W$ is equipped with a map to $\mathbb T_{Y_-} \times \mathbb T_{Y_+}$, and inside the latter we can consider the unstable and stable submanifolds $U_{q_-},S_{q_+}$ of critical points $q_\pm$ of $f_\pm$. 

\begin{lemma}
The moduli space of $L^2_{\mathrm{ext}}$ connections satisfying the perturbed ASD equations above on $W$ asymptotic to $q_\pm$ is oriented diffeomorphic to $\mathbb T_W \cap (U_{q_-} \times S_{q_+})$ by a diffeomorphism preserving the local sytem $\Gamma_\nu$. 
\end{lemma}
\begin{proof}
Identifying $1$-forms on the end $[0,\infty) \times Y_+$ with time-dependent elements $\omega(t)\in\Omega^0(Y_+) \oplus \Omega^1(Y_+)$, the unperturbed ASD equation together with the Coulomb gauge fixing condition
\begin{equation*}
(-d^*)\oplus d^+:\Omega^1\rightarrow \Omega^0\oplus \Omega^+
\end{equation*}
on the cylinder can be written as
$$\frac{d}{dt}\omega(t) + L\omega(t)=0$$
where $$L = \begin{pmatrix} 0 & -d^* \\ -d & *d \end{pmatrix}.$$ The perturbations we consider are defined via the $L^2$-orthogonal projection to harmonic $1$-forms, and the corresponding perturbed ASD equations together with Coulomb gauge fixing then take the form $$\frac{d}{dt}\omega(t) + L\omega(t) + \beta(t) (\mathrm{grad}\tilde f_+)(\omega^{\text{harm}}(t)) = 0,$$where $\tilde f_+: H^1(Y_+; i\mathbb R) \to \mathbb R$ is the periodic lift of the chosen Morse function on $\mathbb T_{Y_\pm},$ which is then applied to the harmonic part of the time-dependent form in $\Omega^1(Y_+)$.
\par
The torus $\mathbb T_W$ can be understood as the space of unperturbed ASD connections in $L^2_k$ modulo gauge on the compact manifold $W$ such that the coclosed parts of the boundary values $$\Pi: L^2_k(\Omega^1(W)) \to \ker(d^*_{Y_-}) \oplus \ker(d^*_{Y_+}) \subset L^2_{k-1/2}(\Omega^1(Y_-)) \oplus L^2_{k-1/2}(\Omega^1(Y_+))$$ lie in the spectral subspace $$H_-^{\le 0} \oplus H_+^{\ge 0} \subset \ker(d^*_{Y_-}) \oplus \ker(d^*_{Y_+})$$ spanned by nonnegative (respectively nonpositive) eigenspaces of the signature operator $*d_{Y_\pm}$. Because the unperturbed ASD connections solves a first order ODE on the ends, each boundary value lying in this subspace extends to a unique solution to the ASD equation on the cylinder (up to gauge); the spectral condition guarantees that the solution is in $L^2_{\text{ext}}$ up to gauge. Notice that $H_+^{\ge 0}$ splits into the sum $H^1(Y_+) \oplus H_+^{>0}$ of the space of harmonic forms and the span of strictly positive eigenspaces, and similarly for $H_-^{\geq0}$.
\par
Because a perturbed ASD connection on $[0, \infty) \times Y_+$ again solves a first order ODE, it is determined by its boundary value in $\ker(d^*_{Y_+})$ (up to gauge). Further, because the operator is only perturbed on the harmonic part, it is still true that for the ASD connection to be $L^2_{\text{ext}}$ this boundary value must lie in the spectral subspace $H^{\ge 0}_+$. A similar argument applies for $Y_-$. Because the equation is unchanged on $W$, the perturbed ASD solutions are identified with a \textit{subset} of $\mathbb T_W$, which was earlier identified with ASD solutions on $W$ whose boundary values lie in $H^{\le 0}_- \times H^{\ge 0}_+$.

Finally, examining the harmonic part, we see that a boundary value extends to a solution to the perturbed equation with the correct asymptotics if and only if the harmonic part $[\omega^{\text{harm}}]|_{\{0\} \times Y_\pm} \in \mathbb T_{Y_\pm}$ lies in the unstable and stable manifold of $q_-$ and $q_+$, respectively. 

The claim about orientations follows from compatibility between orientations and gluing, as in \cite[Section 20.3]{KM}. The claim about the local system follows immediately from the definition: both maps are weighted by a factor of $e^{\int_\nu F^+_{A^t}}$, the only difference being that in the Morse theory case we integrate over the intersection of $\nu$ with the compact part $W$ of the cobordism while in the gauge theory case we integrate over the noncompact surface $\nu^* \subset W^*$. However, because $A^t$ is flat on the ends, the additional contribution is zero.
\end{proof}

This identifies the `top-degree term' in $\bar m_*$, which corresponds to the case where the relevant Dirac operator has index $1$: the kernel of this operator is a one dimensional complex vector space, so that after projectivizing it contributes a single positively-oriented point, up to gauge. Hence the relevant moduli spaces of solutions simply count solutions to the perturbed ASD equations.

To conclude the proof, following \cite[Section $2.8$]{KM} (and dropping local systems for simplicity), the map induced on the Morse complexes
\begin{equation}\label{morsemap}
m_{W}: C_*(\mathbb T_{Y_-})\rightarrow C_*(\mathbb T_{Y_+})    
\end{equation}
by the correspondence $\mathbb T_W \to \mathbb T_{Y_-} \times \mathbb T_{Y_+}$ is obtained (under suitable transversality conditions) by considering triples $(\gamma_-,\gamma_+,w)$ where
\begin{itemize}
\item $\gamma_-:(-\infty,0]\rightarrow \mathbb T_{Y_-}$ is a finite energy half trajectory with $\gamma_-(-\infty) = q_-$
\item $\gamma_+:[0,\infty)\rightarrow \mathbb T_{Y_+}$ is a finite energy half trajectory with $\gamma_+(+\infty) = q_+$ 
\item $w \in \mathbb T_W$ and $i^*_\pm(w)=(\gamma_-(0),\gamma_+(0))$.
\end{itemize}
Because the space of such half-trajectories can be identified with the unstable and stable manifolds $U_{q_-}$ and $S_{q_+}$, we see that $m_W$ is defined by exactly the same counts as the top-filtration term of $\bar m_*$.

\section{$RSF$-spaces} \label{RSF}
We want to take advantage of our description of $\bar{m}_*$ in the preceding section, where we determined the \textit{associated graded} map with respect to an appropriate filtration. To use this to get information about the map $\hat{m}$, we need to make further assumptions about the torsion spin$^c$ $3$-manifold $(Y,\spin)$ under consideration.
\subsection{Definitions.}
To proceed, suppose we have a metric and perturbation on $Y$ (for the spin$^c$ structure $\spin$) such that there are no irreducible solutions, transversality is achieved in the sense of \cite{KM}, and the bar complex is computed by coupled Morse theory as in the previous section. The relatively graded $\Bbb Q$-vector space $\hat{C}_*(Y,\spin)$ can then be identified as a subspace 
\begin{equation}\label{CMfrom-submod}\hat{C}_*(Y,\spin) = \bigoplus_{q \in \text{Crit}(f)} U^{d(q)} \Bbb Q[U] \subset \bigoplus_{q \in \text{Crit}(f)} \Bbb Q[U, U^{-1}] = \bar{C}_*(Y,\spin),
\end{equation}
where the shift $d(q) \in \Bbb Z$ and the identification of $\bar{C}_*(Y,\spin)$ (as a $\Bbb Q[U]$-module) are discussed after Lemma \ref{lemma:Morse-iso}. It is often convenient to normalize $d$ so that $d(q) \le 0$, with largest value equal to zero, and we will do so in what follows.
\par
There are four closely related subtleties. The first is that $\hat C$ is generally \textit{not} a subcomplex: the differential on $\hat C$ is $$\hat{\partial} = -\overline \partial^u_u -  \overline \partial^s_u \partial^u_s,$$ where the term $\partial^u_s$ is defined in terms of \emph{irreducible} Seiberg--Witten solutions on $\mathbb R \times Y$. Secondly, even though $\hat{C}(Y,\spin)$ is naturally a filtered module, with filtration given by (\ref{filtration}), the differential $\hat \partial$ does not necessarily preserve it. Thirdly, the comparison (anti-)chain map 
\begin{equation}\label{connecting-map} p = 1 + \partial_s^u : \hat{C}_*(Y,\spin) \to \bar{C}_*(Y,\spin) 
\end{equation}
also involves this term, and is not necessarily filtration-preserving. Finally, $\hat C$ is generally \textit{not} a $U$-submodule, because $p$ only commutes with the action of $U$ up to homotopy.\footnote{Notice that the claim in the proof of \cite[Proposition 25.1.1]{KM} that the equalities hold at the chain level is incorrect. In our situation with no irreducible critical points, up to overall signs the
chain homotopy between $\bar{m}(U,-)\circ p$ and $p\circ\hat{m}(U,-)$ is 
simply given by $m^u_s(U,-)$; this readily follows from identity $(\mathrm{iv})$ in \cite[Lemma 25.3.6]{KM}.}
\\
\par
Because of this, we make the following definition.
\begin{definition}\label{maindef}
We say a torsion spin$^c$ three-manifold $(Y, \spin_Y)$, where $b_1(Y)\geq 1$, is an \textbf{$RSF$-space} if $Y$ admits a regular choice of metric and perturbations such that:
\begin{enumerate}[label=(\alph*)]
    \item there are only \textit{reducible} solutions to the Seiberg--Witten equations;
    \item the map $\partial_s^u$ is \textit{strictly filtered} with respect to the $U$-filtration on $\hat{C}_*(Y,\spin)$ and $\bar{C}_*(Y,\spin)$. By {strictly filtered} we mean that $$\partial_s^u(\mathcal F_k \hat{C}) \subset \mathcal F_{k-2} \bar{C}$$
for the filtration (\ref{filtration}).
\item the complex $\bar{C}_*(Y,\spin)$ coincides with the coupled Morse complex associated to a Morse function $f:\mathbb{T}_Y\rightarrow \mathbb{R}$ and the corresponding family of (perturbed) Dirac operators as in Section \ref{negative}.
\end{enumerate}
\end{definition}

In this case, $\hat{C}_*(Y,\spin)$ is a filtered complex for which $p$ is a filtered map with associated graded map the natural inclusion; the associated graded complexes are
\begin{align*}
\text{gr}_{2i+1} \hat C &= 0,\\
\text{gr}_{2i} \hat C &= \text{span} \{q \in \text{Crit}(f) \mid d(q) \le i\} \subset C_*(\mathbb T, f; \Gamma_\eta).
\end{align*}
Moving to the cobordism maps, there are two distinct cases to analyze:
\begin{itemize}
\item if $b^+(W) > 0$, after choosing a perturbation by a self-dual form as in \cite[Section $27.2$]{KM}, there are no reducible solutions on the cobordism so that $\bar{m}_*(W)$ vanishes and therefore $\hat{m}_*(W)$ vanishes as well because $p$ is injective.
\item if $b^+(W) =0$, working in the setup of Section \ref{negative} it follows from the equation (\ref{connecting-map}) that the associated graded map of $\bar{m}_*(W)$ determines the associated graded map of $\hat{m}_*(W)$, and we determined the former in Theorem \ref{thm:constant-term}.
\end{itemize}
We record these observations as a lemma in the specific case of interest, where the two ends coincide.
\begin{lemma}\label{lemma:CM-from-map}
    If $(Y, \spin)$ is an $RSF$-space, the complex $\hat{C}_*(Y, \spin_Y; \Gamma_\eta)$ is a filtered complex. If $(W, \spin_W, \nu)$ is a cobordism from $(Y, \spin, \eta)$ to itself, the map $\hat{m}_*(W)$ satisfies $$\textup{gr } \hat{m}_*(W, \spin_W, \nu) = U^d m_{W, \nu}$$ for $b^+(W) = 0$, and is zero on the nose otherwise.
\end{lemma}

\subsection{Examples of $RSF$-spaces}\label{exman}
The simplest example of an $RSF$-space is given by a torsion spin$^c$ three-manifold $(Y,\spin)$ admitting a metric of positive scalar curvature, as follows from the discussion in \cite[Chapter $36$]{KM}. These examples are somewhat trivial because the family of Dirac operators parametrized by $\mathbb{T}_Y$ never has kernel; in the setup of Theorem \ref{thm:weak-calculation}, which we prove below, this leads to \textit{vanishing} results. We now discuss some non-trivial examples where we obtain instead \textit{rigidity} results.
\par
All of our non-trivial examples come from the following procedure: 

\begin{enumerate}[label=(\roman*)]
\item Find a metric $g$ so that $(Y, \spin)$ supports no irreducible Seiberg--Witten solutions, so that the critical locus of the Chern--Simons--Dirac functional $\mathcal{L}$ is exactly $\mathbb{T}$. Notice that $\mathcal{L}$ fails to be Morse--Bott exactly along the locus $\mathsf{K}$ where the corresponding Dirac operator has kernel.
\item After a suitable perturbation without introducing irreducible solutions, choose an additional Morse function $f: \mathbb T_Y \to \mathbb R$ so that the critical points of $f$ are Dirac operators with no kernel. Labeling the eigenvalues of the Dirac operator at each critical point $q$ as following Lemma \ref{lemma:Morse-iso} and writing $d(q)$ for the index of the first positive eigenvalue, we demand that for each pair of critical points $x,y$ with $d(x) > d(y)$, we have $f(x) < f(y)$. (This includes as a special case the `$A$-adapted perturbations’ of \cite{spectrally-large}).
\item Even if the unperturbed equations did not admit irreducible solutions, it is challenging to carry out the perturbation process without introducing such solutions. One uses the geometry at hand to argue that the perturbed Seiberg--Witten equations are now regular, but still have no irreducible solutions. After this, one can add additional perturbations as in \cite[Ch. 12]{KM} to achieve transversality (in particular by making the spectrum of the pertubed Dirac operators at reducible critical points simple).
\end{enumerate}

\begin{lemma}
For any metric and perturbation on $(Y, \spin)$ constructed as in the above procedure, the map $\partial_s^u$ is strictly filtered.
\end{lemma}
\begin{proof}
The unstable critical points $(x, \lambda_i(x))$ have $i < d(x)$, while the stable critical points $(y,\lambda_j(y))$ have $d(y) \le j$. Suppose $\langle \partial_s^u (x,\lambda_i(x)), (y,\lambda_j(y))\rangle \ne 0$, so there exists an irreducible flowline from the stable critical point $(x,\lambda_i(x))$ and flowing to the unstable critical point $(y,\lambda_j(y))$. 

The perturbed Chern--Simons--Dirac functional coincides with $\varepsilon f$ on the reducible critical set, so we must have $f(x) > f(y)$ and therefore $d(x) \le d(y)$. Therefore $i < d(x) \le d(y) \le j$ and we see that $\partial_s^u$ is strictly filtered.
\end{proof}

We now list some examples of $RSF$-spaces obtained from this procedure.

\begin{example}\label{ex:torus}
The simplest non-trivial example is that of the torus endowed with a flat metric as discussed \cite[Section 38]{KM}. In this setup, the quantity $d(x)$  in (\ref{CMfrom-submod}) is given by $$d(x) = \begin{cases} 0 & \mathrm{ind}(x) >0  \\ -1 & \mathrm{ind}(x) = 0, \end{cases}$$ 
and the complex $\hat{C}_*(T^3,\spin_0)$ is given by
\[\begin{tikzcd}
	{\mathbb{Z}[U]} \\
	& {\mathbb{Z}[U]^{\oplus 3}} && {U^{-1}\mathbb{Z}[U]} \\
	&& {\mathbb{Z}[U]^{\oplus 3}}
	\arrow["U^{-1}", from=1-1, to=2-4]
\end{tikzcd}\]
where the $i$th column corresponds to the points with Morse index $3-i$, and the arrow is an isomorphism.
\end{example}
\begin{example}
In \cite[Section 37.4]{KM} the authors show also that the flat three-manifolds $Y$ with $b_1 = 1$ are $RSF$-spaces. For all but one spin$^c$ structure on $Y$, the Morse function can be chosen perfect with $d(x_1) = d(x_0) = 0$. For one spin$^c$ structure $\mathfrak s_0$, the Morse function can be chosen to have four critical points ($x_0, x_1$ of index $0$ and $y_0, y_1$ of index $1$) and $d(x_0) = d(x_1) = d(y_0) = -1$ while $d(y_1) = 0$. Notice that the flat 3-manifolds are precisely the mapping tori of finite-order diffeomorphisms of $T^2$. 
\end{example}

\begin{example}\label{spectral}
In \cite{spectrally-large} it is shown that for $g=2,3$, $S^1\times \Sigma_g$ is an $RSF$-space (for the unique torsion spin$^c$ structure).
\par
The proof uses ideas from spectral geometry. We say that a Riemannian $3$-manifold $Y$ is \textit{spectrally large} if the first eigenvalue of the Hodge Laplacian on coexact $1$-forms is large compared to the curvature (in a suitable quantitative sense). Under this hypothesis, one can then add perturbations of \textit{spinorial type} so that the equations do not have irreducible solutions, and the hypersurface $\mathsf{K}\subset \mathbb{T}_Y$ consisting of the locus where the Dirac operator has kernel is transversely cut out in the space of operators. Notice that the latter does not imply that $\mathsf{K}$ is smooth when $b_1\geq 4$. Now, under the technical assumption that for the spin$^c$ structure $\spin$ the hypersurface $\mathsf{K}$ is smooth, one can find a Morse perturbation satisfying the conditions of Definition \ref{maindef}; in particular $(Y,\spin)$ is an $RSF$-space. 

For a handful of $g$ including $g = 2,3$, the manifold $S^1 \times \Sigma_g$ can be shown to admit a spectrally large metric. For $g \le 3$ the hypersurface $\mathsf{K}$ can be taken to be smooth after perturbation; for $g \ge 4$, however, this cannot be done.
\end{example}

\begin{example}\label{mapping}
In \cite{LinTheta} it is shown that if $Y$ is the mapping torus of a finite order diffeomorphism $\varphi$ of a surface $\Sigma$ of genus $\geq 2$ with $\Sigma/\varphi=\mathbb{P}^1$, and $\spin$ is self-conjugate, then $(Y,\spin)$ is an $RSF$-space. In this context, $b_1=1$ and the locus $\mathsf{K}\subset\mathbb{T}_Y$ where the Dirac operator has kernel is an arbitrary conjugation-symmetric set of points (with multiplicity). So the critical set will be an even set of points and $d(x)$ will increase by the multiplicity of $\mathsf{K}$ as we pass between adjacent degree zero and degree $1$ critical points. As a consequence, one can obtain concrete examples for which the reduced Floer homology $\HMr_*(Y, \spin; \Gamma_\eta)$ has arbitrarily many $U$-torsion summands, with arbitrarily large length.
\end{example}

\subsection{Proof of Theorem \ref{thm:weak-calculation}}
In this section and what follows, we will frequently use the Mayer--Vietoris type sequence \begin{equation}\label{eq:MV1}
 0 \to H_3(W) \xrightarrow{p_3} H_3(X) \xrightarrow{\delta_3} H_2(Y) \xrightarrow{i_2 - j_2} H_2(W)\xrightarrow{p_2}  H_2(X)\xrightarrow{\delta_2} \cdots
\end{equation}
as well as the sequences which are related to this one by Poincar\'e duality and chain-level duality
\begin{align*}
    \cdots &\to H_3(W, \partial W) \to H_2(Y) \to H_2(X) \to H_2(W, \partial W) \to H_1(Y) \to \cdots\\
    0 &\to H^1(W, \partial W) \to H^1(X) \to H^1(Y) \to H^2(W, \partial W) \to H^2(X) \to \cdots \\
    \Bbb Z &\to H^1(X) \to H^1(W) \to H^1(Y) \to H^2(X) \to H^2(W) \to \cdots.
\end{align*}
It is useful to keep in mind the intersection-theoretic definition of $\delta_k$: one takes a generic $k$-cycle in $X$ and intersects it with $Y$ to give a $(k-1)$-cycle in $Y$. This is Poincar\'e dual to the map $$f^*_{4-k} : H^{4-k}(X) \to H^{4-k}(Y)$$ given by pullback along the inclusion.

Observe now that because $b^+(X) > 0$ and $b^+(W) = 0$, the map $p_2: H_2(W) \to H_2(X)$ is not surjective, so the exactness of (\ref{eq:MV1}) implies that there exists some $[\bar \nu] \in H_2(X;\Bbb R)$ for which $$[\eta] = [\bar \nu \cap Y] \ne 0 \in H_1(Y; \Bbb R),$$ and we choose chain-level representatives for these classes. Applying Proposition \ref{prop:gluing}, we find $$\mathfrak m(X, \spin_W) = \lim_{t \to 0} \mathfrak m(X, \spin_W, t[\bar \nu]) = \lim_{t \to 0} \textup{Tr}\big(\HMr_*(W, \spin_W; \Gamma_{t\nu})\big).$$
If $b^+(W) > 0$ it follows from Lemma \ref{lemma:CM-from-map} that this quantity is zero. When $b^+(W) = 0$, we will relate this to the associated graded map discussed above by way of the classic Hopf trace formula: if $C$ is a finite-dimensional $\Bbb Z/2$-graded chain complex, then the alternating trace of a map $\phi: C \to C$ coincides with the alternating trace of the induced map on homology $\phi_*$. We will also use the elementary observation that the trace of a filtered map is the same as the trace of its associated graded map, because the trace of an upper-triangular matrix is independent of the entries above the diagonal. 

\begin{proof}[Proof of Theorem \ref{thm:weak-calculation}]
In our setting the differential $\hat{\partial}$ and the map $p$ are filtered by definition, and this can be applied as follows. Observe that when $t\eta \ne 0$, the Morse homology $H_*(\mathbb T_Y, f; \Gamma_{t\eta})$ is trivial. Recall that we normalize the quantity $d(q)$ from (\ref{CMfrom-submod}) so that $d(q) \le 0$; say $d(q) = k \le 0$ is the minimum value. To have a concrete example in mind, in the case of the torus $T^3$ in Example \ref{ex:torus}, we have $$\text{gr}_i \hat{C}(T^3, \spin_0; \Gamma_\eta) = \begin{cases} C_*(T^3, f; \Gamma_\eta) & i \ge 0 \text{ and is even} \\ \langle x_0\rangle & i = -2 \\ 0 & \text{otherwise} \end{cases}$$ Here $k = d(x_0) = -1$ is the minimum value. 

In particular, the associated graded complex is $C_*(\mathbb T_Y, f; \Gamma_{t\eta})$ for $i \ge 0$ and zero for $i < 2k$, so we have $E^2_i = 0$ in the associated spectral sequence converging to $\HMr_*(Y, \spin; \Gamma_{t \eta})=\HMf_*(Y, \spin; \Gamma_{t \eta})$ except possibly when $i$ is an even integer satisfying $2k \le i \le -2$. In particular, $E^2$ and all successive pages of the associated spectral sequence are finite-dimensional. Applying Hopf's trace formula to the successive pages of the spectral sequence, we find
\begin{align*}
    \mathfrak m(X, \spin_W) &= \lim_{t \to 0} \text{Tr}\left(HM_*(W, \spin_W; \Gamma_{t\nu})\right) = \lim_{t \to 0} \sum_{i=k}^{-1} \text{Tr}(E^\infty_{2i}(\hat{m}(W, \spin_W; \Gamma_{t \nu})\big) \\
    &= \lim_{t \to 0} \sum_{i=k}^{-1} \text{Tr}(E^r_{2i}(\hat{m}(W, \spin_W; \Gamma_{t \nu})\big) 
\end{align*}
for all $r \ge 2$, where $$E^r_{2i}(\hat{m}(W, \spin_W; \Gamma_{t \nu})):E^r_{2i}(\hat{C}(Y, \spin; \Gamma_{t \eta}))\to E^r_{2i}(\hat{C}(Y, \spin; \Gamma_{t \eta}))$$ is the map induced on the $E^r$ page of the spectral sequence. 

The $E^2$ page of the spectral sequence coincides with the homology of the associated graded complex $\text{gr}_{2i} \hat{C}(Y, \spin_W; \Gamma_{t \eta})$; notice that for each $i$ this can be naturally identified with a subcomplex of $C_*(\mathbb T_Y, f; \Gamma_{t \eta})$ which, as an $R$-module, is independent of $t$.
Applying Hopf's formula once more, we see $$\mathfrak m(X, \spin_W) = \lim_{t \to 0} \sum_{i=k}^{-1} \text{Tr}\left(U^d m_{W, t\nu} \lvert_{\text{gr}_{2i} \hat{C}(Y, \spin; \Gamma_{t \eta})}\right)$$ for $d$ as in Theorem \ref{thm:constant-term}. Because the domain is independent of $t$ and the map is continuous in $t$, taking the limit as $t \to 0$ gives 
\begin{equation}\label{eqn:trace-formula}\mathfrak m(X, \spin_W) = \sum_{i=k}^{-1} \text{Tr}\left(U^d m_W \lvert_{\text{gr}_{2i} \hat{C}(Y, \spin)}\right).\end{equation}

From this we immediately conclude our rigidity result that $\mathfrak m(X, \spin_W)$ depends only on $m_W, d$, and $\text{gr}_{2i} \hat{C}(Y, \spin)$, because the map $m_W$ depends only on the diagram (\ref{eq:correspondence}), the integer $d$ depends only on the characteristic numbers of $W$ and the spin$^c$ structure $\spin_W$, and the complex depends only on $(Y, \spin)$. 

We have proved much of Theorem \ref{thm:weak-calculation}; all that remains is to show that this trace vanishes when the three listed conditions are not met. That the Seiberg--Witten invariant $\mathfrak m(X, \spin_W)$ vanishes for $b^+(W) > 0$ was discussed above. Because $m_W$ has degree $b_1(W) - b_1(Y)$, we see that even when $b^+(W) = 0$ the trace can only be nonzero if $b_1(W) = b_1(Y)$. Similarly, this trace can only be nonzero if $d = 0$, which when $b_1(W) = b_1(Y)$ is equivalent to $d(\spin_W) = 0$.
\end{proof}

It is worth naming the dependency of the constant in \eqref{eqn:trace-formula}.

\begin{definition}
Suppose $(X, Y, \mathfrak s_Y)$ satisfy the hypotheses of Theorem \ref{thm:weak-calculation}. We write \[c(W, Y, \spin_Y) = \sum_{i=k}^{-1} \text{Tr}\left(m_W \lvert_{\text{gr}_{2i} \hat{C}(Y, \spin)}\right)\] for the quantity arising in \eqref{eqn:trace-formula} when $d = 0$.
\end{definition}

\begin{remark}\label{rmk:sign-ambiguity}
    There is a sign ambiguity in the definition of $c(W, Y, \spin_Y)$, as the map $m_W$ depends on a choice of a homology orientation for $W$. In particular, the sign ambiguity is independent of $\spin_Y$; summing $c(W, Y, \spin_Y)$ over all $\spin_Y$ gives an integer, well-defined up to sign.
\end{remark}

\section{Examples} \label{c-comp}
Theorem \ref{thm:weak-calculation} is useful if we can actually compute this quantity $c(W,Y, \spin)$. We do this in several examples. Throughout this section, we suppose $(Y, \spin)$ is an RSF-space and $W: Y \to Y$ is a cobordism with $b^+(W) = 0$ and $b_1(W) = b_1(Y)$. Notice that the latter condition allows us to define the discriminant as in (\ref{discriminant})
\par
When $Y$ has positive scalar curvature and $b_1(Y) > 0$, the family of Dirac operators on $\mathbb T_Y$ has no kernel, so $d(q) = 0$ for all critical points $q$ and therefore $HM_*(Y, \spin; \Gamma_\eta)$ is zero for any spin$^c$ structure on $Y$. Of course, it follows that $c(W,Y, \spin) = 0$. 
\par
As a less tautological example, we compute the quantity $c(W,Y, \spin_0)$ in Theorem \ref{thm:weak-calculation} for $\spin_0$ the torsion spin$^c$ structure on $T^3$. 

\begin{prop}\label{prop:torus-calc}
For $Y = T^3$, the quantity $c(W,Y, \spin_0)$ in Theorem \ref{thm:weak-calculation} is equal up to sign to the natural number $\textup{disc}(W)$ in (\ref{discriminant}).
\end{prop}

\begin{proof}
It was established in the proof of Theorem \ref{thm:weak-calculation} that the quantity $c(W,Y, \spin_0)$ is equal to $$\text{Tr}\left(m_W | \langle x_0\rangle\right).$$ Because the Morse function is perfect, the map $m_W$ is the same as the map it induces on homology, which is computed in \cite[Section 2.8]{KM} to be the composition $$(r_+)_* \circ PD^{-1}_{\mathbb T_W} \circ r_-^* \circ PD_{\mathbb T_Y},$$ where $$T^3 \xleftarrow{r_-} T^3 \xrightarrow{r_+} T^3$$ is the diagram induced by the pair of inclusions  $i,j:Y\rightarrow W$. Generally, if $f: T^3 \to T^3$ is any linear map, we have $$f_*(x) = \begin{cases} x & |x| = 0 \\ \det(f) x & |x| = 3, \end{cases}$$ and similarly for $f^*(x)$. The computation in middle degrees is somewhat more subtle; the map in degree $k$ can be computed in terms of the matrix of $k \times k$ minors of $f$. 

Thus, we have \begin{align*}m_W(x_0) &= \pm (r_+)_* \; (PD^{-1}_{\mathbb T_W}) \; (r_-^*)(x_3) = \pm (r_+)_* \; (PD^{-1}_{\mathbb T_W}) \; (\det(i) x_3) \\
&= \pm (r_+)_* \; (\det(i) x_0) = \pm \det(i) x_0.\end{align*} The discriminant in (\ref{discriminant}) is, by definition, $\det(i)$. 
\end{proof}

\begin{remark}
    This argument is where the sign ambiguity in Theorem \ref{thm:torus-rigidity} appears. Homology orientations on $Y$ and $W$ are used to pin down the sign on the Poincar\'e duality maps above, which pins down the sign in the definition of $\mathfrak m(X)$ as well. We do not see the appeal in working out the homology orientations in detail, so choose to record the result up to sign. The sign ambiguity is the same as in Remark \ref{rmk:sign-ambiguity}.
\end{remark}

This argument goes through with minimal change for the other flat $3$-manifolds with $b_1(Y) = 1$. For all but one spin$^c$ structure $\spin_Y$, the Dirac operator on $Y$ has no kernel, and hence like in the example of PSC manifolds, spin$^c$ classes on $X$ restricting to something other than $\spin_Y$ contribute nothing to the computation of $\mathfrak m(X)$. For the following proposition, write $\spin_0$ for the unique spin$^c$ structure which admits spin$^c$ connections with parallel spinors.

\begin{prop}\label{prop:flat-calc}
If $(Y, \spin_0)$ is a flat 3-manifold with $b_1(Y) = 1$ equipped with its canonical spin$^c$ structure, then the quantity $c(W,Y, \spin_0)$ in Theorem \ref{thm:weak-calculation} is equal up to sign to the natural number $\textup{disc}(W)$ defined as in (\ref{discriminant}).
\end{prop}
Notice that by (\ref{eq:MV1}) in order for the glued up manifold $X$ to have $b^+\geq1$ we have to assume that the maps (\ref{eq:correspondence}) coincide or, equivalently, that $H_2(Y) \to H_2(X)$ is injective.

\begin{proof}
As discussed above, the Morse function $f: S^1 \to \Bbb R$ has four critical points $x_0, x_1, y_0, y_1$ with $d(x_0) = d(x_1) = d(y_0) = -1$ and $|x_i| = 0$ while $|y_i| = 1$; the associated graded complex is only relevant for $i = -1$, in which case we obtain $\langle x_0, x_1, y_0\rangle$. The trace of $m_W$ on this subcomplex is the negative of the trace of $m_W$ on the quotient complex $\langle y_1\rangle$, where it can be computed as in Proposition \ref{prop:torus-calc} to equal $\pm \det(i^*)$. (In fact, the map $\bar m_*$ must be given by scaling by $\det(i^*)$ on every generator.)
\end{proof}

More generally, we have the following.

\begin{prop}\label{prop:case-i-equals-j}
Suppose the maps in (\ref{eq:correspondence}) coincide. Then $$c(W,Y,\spin) = \pm \textup{disc}(W)\cdot\chi(HM_*(Y, \spin; \Gamma_\eta)),$$
where again $\textup{disc}(W)$ is defined as in (\ref{discriminant}).
\end{prop}
\begin{proof}
Recall that we defined \[c(W,Y,\spin) = \sum_{i=k}^{-1} \text{Tr}\left(U^d m_W \lvert_{\text{gr}_{2i} \hat{C}(Y, \spin)}\right),\] summing only over those indices for which $\text{gr}_{2i} \hat{C}(Y, \spin, \Gamma_{t\eta})$ is not acyclic. Because the associated graded complex of $\hat{C}(Y,\spin, \Gamma_{t\eta})$ is independent of $t$, the Euler characteristic of each summand is independent of $t \in \mathbb R$; summing from $k$ to $-1$, the result is equal to $\chi(HM_*(Y, \spin; \Gamma_\eta))$. We will prove that the map $m_{W}: C_*(\mathbb T, f) \to C_*(\mathbb T, f)$ is scalar multiplication by $\pm \det(i^*) = \text{disc}(W)$, so the same is true of the restriction to each $\text{gr}_{2i} \hat{C}$, giving the stated result.

Because both endpoint maps in the correspondence $$\mathbb T_Y \leftarrow \mathbb T_W \to \mathbb T_Y$$ are the same (say, $i: \mathbb T_W \to \mathbb T_Y$) by assumption, in the equation $$r(w) = (\gamma_-(0), \gamma_+(0))$$ we are counting Morse trajectories together with an element $w \in \mathbb T_W$ mapping to $\gamma_-(0)=\gamma_+(0)$. Because we are counting trajectories between points of the same Morse index, we must have that $\gamma_\pm$ are constant trajectories at the same critical point $q \in \text{Crit}(f)$. Thus we are computing the signed number of $w \in \mathbb T_W$ which map to $q \in \mathbb T_Y$, which coincides with $\det(i^*)$ up to sign.
\end{proof}

Looking back at Section \ref{exman}, for $Y=T^3$ the Euler characteristic term is $\pm1$, while the Euler characteristic term for $S^1 \times \Sigma_g$ is $\pm 2$ for $g=2$ and $\pm 6$ for $g=3$. The mapping tori of Example \ref{mapping} provide examples of manifolds for which the Euler characteristic term is arbitrarily large.
\\
\par
A particularly interesting example is the case of $Y=S^1\times \Sigma_2$. Write $\varphi: S^1 \times \Sigma_2 \to S^1 \times \Sigma_2$ for a diffeomorphism of the form $\mathrm{id}_{S^1} \times \phi_{\Sigma_2}$. 

\begin{prop}
    Suppose $(Y, \spin_0)$ is $S^1 \times \Sigma_2$ with its unique torsion spin$^c$ structure. If $(W, \spin_W)$ is a self-cobordism of $(Y, \spin_0)$ for which $i^* = j^*$ and $X$ is obtained by gluing the ends together using $\varphi$, then $$ \mathfrak m(X, \spin_W) = \pm\textup{disc}(W)\cdot(\mathrm{Tr}(\phi_*)-2)$$where $\phi_*$ denotes the action on $H_1(\Sigma_2)$.  
\end{prop}

\begin{proof}
Choose $\eta$ to be a non-zero multiple of $S^1$. We have $i^* [\eta] = \varphi^* j^* [\eta]$ so $\eta$ extends to a 2-chain $\bar \nu$ on $X$ as in Proposition \ref{prop:gluing}. For this $\eta$ we have
\begin{equation*}
\mathit{HM}_*(S^1\times \Sigma_2,\spin_0;\Gamma_{\eta})\cong H_*(\Sigma_2;\mathbb{R})
\end{equation*}
as $\mathbb{Z}_2$-graded vector spaces. Explicitly, following \cite[Corollary 4.5]{spectrally-large}, we consider the Abel--Jacobi map $AJ: \Sigma_2 \to \text{Jac}(\Sigma_2)$, which is an embedding. Identifying $\mathbb T_Y \cong S^1 \times \text{Jac}(\Sigma_2)$, there exists a perfect Morse function $f: \mathbb T_Y \to \mathbb R$ as in Section \ref{exman}, for which $\mathsf{K}$ is a sphere bundle around the image of $AJ$. We have $d(x) = 0$ for critical points in one component of the complement and $d(x) = -1$ in the other. The latter component is diffeomorphic to $\Sigma_2 \times D^3$, and the Morse function may be chosen to have exactly six critical points in this component. The isomorphism \[\mathit{HM}_*(S^1\times \Sigma_2,\spin_0;\Gamma_{\eta}) \cong H_* \text{gr}(\widehat{CM}_*(S^1 \times \Sigma_2, \spin_0; \Gamma_\eta)) \cong H_*(\Sigma_2;\mathbb R)\] then arises because the associated graded complex $\text{gr}_{-2}$ is equal to $C_*(\Sigma_2;\mathbb R)$, while all other associated graded terms are acyclic. 

Because the Morse function $f$ is perfect, the cobordism map induced by the pair $(1, \varphi^*)$ on the chain level coincides with the map the correspondence of tori induces on homology, which coincides with $\varphi_*$. Because the map induced by AJ on homology commutes with the action of the mapping class group, we see the restriction to $C_*(\Sigma_2;\mathbb R)$ also coincides with $\varphi_*$. The map induced by the pair $(i^*, j^*)$ is $\text{disc}(W)$ by Proposition \ref{prop:case-i-equals-j}, so the pair $(i^*, \varphi^* j^*)$ induces $\varphi_*$ scaled by $\text{disc}(W)$. The stated formula follows.
\end{proof}
As a concrete example, consider the mapping class of $\Sigma_2$ obtained as the connect sum of the identity on $T^2$ and the Anosov map induced by the matrix
\begin{equation*}
\begin{bmatrix}
2&1\\
1&1
\end{bmatrix}.
\end{equation*}
Then $X$ is a homology four-torus with determinant $1$, and we obtain that $$\sum_{\spin\lvert_Y=\spin_0}\mathfrak{m}(X,\spin)=\pm3.$$ In fact, $X$ is obtained from $T^4$ by doing knot surgery \cite{FintushelStern} along a standard $T^2$ using the figure eight knot, and the value of $3$ we found corresponds to the constant term in its Alexander polynomial. The non-constant terms correspond to spin$^c$ structures which are non-torsion over $Y$, which our techniques do not have access to.

\section{Non-separating three-tori}\label{torus}
The main goal of this section is to prove Theorem \ref{thm:torus-rigidity} and Theorem \ref{cor:det-squared}, and discuss some concrete examples. 

\subsection{Proof of Theorem \ref{thm:torus-rigidity}}
Our first observation restricts attention to the case we can actually calculate, namely that of spin$^c$ structures on $X$ which are torsion on $Y$.

\begin{lemma}\label{lemma:Thurston}
Suppose $X$ is a closed, oriented, connected $4$-manifold with $b^+(X) \ge 2$, and that $X$ contains a non-separating hypersurface $Y$ with trivial Thurston norm. If $\mathfrak m(X, \spin_X) \ne 0$, then the restriction $\spin_X|_Y$ is torsion.
\end{lemma}
\begin{proof}
Let $\Sigma \subset Y$ be a torus, so in particular $[\Sigma] \cdot [\Sigma] = 0$ and $g(\Sigma) = 1$. First, the adjunction inequality \cite[Theorem 40.2.3]{KM} implies that if $\mathfrak m(X, \spin_X) \ne 0$, we have $$0 = 2g - 2 \ge |\langle c_1(\spin_X), [\Sigma]\rangle| + [\Sigma] \cdot [\Sigma] = |\langle c_1(\spin_X), [\Sigma]\rangle|.$$ It follows that $$0 = \langle c_1(\spin_X), [\Sigma]\rangle = \langle c_1(\spin_X)|_Y, [\Sigma]\rangle$$ for all tori $\Sigma$ in $Y$. Because $Y$ has trivial Thurston norm, the real homology $H_2(Y; \Bbb R)$ is generated by embedded tori, and it follows that $c_1(\spin_X)|_Y$ pairs trivially with every real homology class. By Poincar\'e duality, $c_1(\spin_X)|_Y$ is torsion.
\end{proof}

Now to prove Theorem \ref{thm:torus-rigidity}, assume $Y = T^3$. For a spin$^c$ structure on $W$ with $i^* \spin_W \cong j^* \spin_W \cong \spin_0$ the torsion spin$^c$ structure on $Y$, we established in Proposition \ref{prop:torus-calc} that $$c(W,Y,\spin_0) = \begin{cases} \pm \text{disc}(W) & \text{if }b^+(W) = 0, \; b_1(W) = b_1(Y), \; \text{and}\; d(\spin_W) = 0, \\ 0 & \text{otherwise}. \end{cases}$$
Applying Lemma \ref{lemma:Thurston}, re-indexing the sum, and applying Theorem \ref{thm:weak-calculation} gives

\begin{align*}
    \mathfrak m(X) &= \sum_{\spin_X} \mathfrak m(X, \spin_X) = \sum_{\substack{\spin_X \\ \spin_X|_Y \text{ torsion}}} \mathfrak m(X, \spin_X) = \sum_{\substack{\spin_W \\ i^* \spin_W \cong j^* \spin_W \\ i^* \spin_W \cong \spin_0}} \sum_{\substack{\spin_X \\ p^* \spin_X \cong \spin_W}} \mathfrak m(X, \spin_X) \\
    &= \sum_{\substack{\spin_W \\ i^* \spin_W \cong j^* \spin_W \\ i^* \spin_W \text{ torsion} \\ d(\spin_W) = 0}} c(W, Y, \spin_0).
\end{align*}
By Proposition \ref{prop:torus-calc}, these summands simplify to $\pm D(W) \text{disc}(W)$ when $b^+(W) = b_1(W) - b_1(Y) = 0$, and is otherwise zero; further, as discussed in Remark \ref{rmk:sign-ambiguity}, the equality $c(W, Y, \spin_Y) = D(W) \text{disc}(W)$ holds up to an \textit{overall} sign ambiguity.


\begin{remark}
    More generally, if $(Y, \spin)$ is an $RSF$-space with trivial Thurston norm, the combination of Lemma \ref{lemma:Thurston} and Theorem \ref{thm:weak-calculation} allow to completely determine $\mathfrak m(X)$ if one could compute the quantities $c(W,Y, \spin)$ in the case of interest.
\end{remark}

\subsection{Homology 4-tori}
Let $X$ be a homology 4-torus. The goal of this subsection is to deduce Theorem \ref{cor:det-squared} from Theorem \ref{thm:torus-rigidity} via computations in algebraic topology. Recall the definition of the determinant of $X$ in (\ref{determinant}); this is the class of the 4-form $\cup^4_X \in \Lambda^4 H^1(X;\Bbb Z) \cong \Bbb Z$ up to sign. We begin with the following lemma.

\begin{lemma}\label{lemma:det1}
Suppose $X$ is a rational homology $4$-torus and $Y \subset X$ is a non-separating $3$-torus. Then we have $\det(X) \ne 0$ if and only if the intersection map$$\delta_3: H_3(X) \to H_2(Y)$$has rank three, in which case $\det(X)$ is the index of $\textup{im}(\delta_3)$ in $H_2(Y)$.
\end{lemma}
\begin{proof}
    Because the map $\delta_3$ is given by intersection with $Y$, if we choose a basis for $H_3(X)$ given by $[Y], \Sigma_1, \Sigma_2, \Sigma_3$, then $$\det(X) = [Y] \cap \Sigma_1 \cap \Sigma_2 \cap \Sigma_3 = \delta_3(\Sigma_1) \cap \delta_3(\Sigma_2) \cap \delta_3(\Sigma_3).$$ Choose $\Sigma_1, \Sigma_2, \Sigma_3$ and basis vectors $S_i$ of $H_2(Y)$ so that the map $\delta_3: H_3(X) \to H_2(Y)$ is in Smith normal form, with $\delta_3(\Sigma_i) = a_i S_i$. Then because $\det(T^3) = 1$ we have $\det(X) = a_1 a_2 a_3$. This is zero if and only if $\text{rank}(\delta_3) < 3$; if $\text{rank}(\delta_3) = 3$, this quantity coincides with the index of the image of $\delta_3$. 
\end{proof}


We will also be interested in the behavior of $\delta_2: H_2(X) \to H_1(Y)$. Write $c_i = 3 - \text{rank}(\delta_i)$. It is straightforward to verify using the Mayer--Vietoris type sequence (\ref{eq:MV1}), Poincar\'e duality, and the fact that the rank of homology and cohomology coincide, that $$b_4(W) = 0, \quad b_3(W) = 1 + c_3, \quad b_2(W) = 3 + c_2 + c_3, \quad b_1(W) = 3 + c_2, \quad  b_0(W) = 1.$$ We also have $b_i(W, \partial W) = b_{4-i}(W).$ 

\begin{lemma}
If $X$ is a rational homology torus with $Y \subset X$ a non-separating 3-torus, then $b^+(W) = c_2$. In particular, $b^+(W) = 0$ if and only if $b_1(W) = 3$.
\end{lemma}
\begin{proof}
Consider the relative long exact sequence of the pair $(W, \partial W)$, we have $$0 \to \Bbb Z \to H_3(W) \to H_3(W, \partial W) \to H_2(Y) \oplus H_2(Y) \to H_2(W) \xrightarrow{\varphi} H_2(W, \partial W).$$ We computed the rank of each group in the sequence above in terms of $c_2$ and $c_3$; starting on the left, we can compute the rank of each map in the sequence, concluding that $\text{rank}(\varphi) = 2c_2$. 

As remarked in the introduction, the intersection form on $\hat H(W) \cong \text{im}(\varphi)/\text{tors}$ is well-defined and nondegenerate, and is the intersection form associated to $W$. On the other hand, by additivity of signature under boundary-gluing, we have $\sigma(W) = \sigma(X) = 0$. Because $b^+(W) = b^-(W)$ it follows that $$b^+(W) = \frac 12 \dim \hat H(W) = \frac 12 \text{rank}(\varphi) = c_2.\eqno\qedhere$$
\end{proof}

Recall the definition of discriminant in (\ref{discriminant}).

\begin{lemma}\label{lemma:det2}
    If $X$ is a rational homology torus with $Y \subset X$ a non-separating 3-torus and $W = X- Y$ its complement, then $\det(X) \ne 0$ if and only if $b_1(W) = 3$ and $\textup{disc}(W) \ne 0$. In this case, $W$ is a rational homology cobordism from $Y$ to itself, the maps (\ref{eq:correspondence}) are equal, and $\det(X) = \textup{disc}(W)$.
\end{lemma}
\begin{proof}
Suppose first that $\det(X) \ne 0$. Then $\delta_3$ has rank three. Using the commutative diagram \[\begin{tikzcd}
	{\Lambda^2 H_3(X)} && {\Lambda^2 H_2(Y)} \\
	{H_2(X)} && {H_1(Y)}
	\arrow["{\Lambda^2(\delta_3)}"', from=1-1, to=1-3]
	\arrow["{\cup^2_X}", from=1-1, to=2-1]
	\arrow["\cong", from=1-3, to=2-3]
	\arrow["{\delta_2}", from=2-1, to=2-3]
\end{tikzcd}\]
and the fact that $\Lambda^2(\delta_3)$ also has rank three, it follows that $\delta_2$ has rank three as well, so that $b_1(W) = 3$.

Next, suppose $b_1(W) = 3$ and consider the composition $$H^1(X) \xrightarrow{p^*} H^1(W) \xrightarrow{i^*} H^1(Y).$$ The composite $(pi)^* = f^*$ is Poincar\'e dual to the map $\delta_3$. The first map has rank three, as $$0 \to \Bbb Z \to H^1(X) \xrightarrow{p^*} H^1(W) \to H^1(Y) \to \cdots$$ is exact and $b_1(X) = 4$, so the composite $PD(\delta_3) = i^* p^*$ has rank three if and only if $i^*$ has rank three. This establishes the claim that $\det(X) \ne 0$ if and only if $b_1(W) = 3$ and $\det(i^*) \ne 0$.
\par
Supposing now that $b_1(W) = 3$ and $\det(i^*) \ne 0$, because $\det(X) \ne 0$ we find that $W$ has the Betti numbers of the 3-torus; as $i^*: H^1(W; \Bbb Q) \to H^1(Y; \Bbb Q)$ is an isomorphism on rational cohomology, taking the cup-square and cup-cube of this map we find the natural map $i^*: H^k(W; \Bbb Q) \to H^k(Y; \Bbb Q)$ is an isomorphism for all $k$. The same argument holds for $j^*$. 

For the final claim, observe that exactness of $$0 \to \Bbb Z \to H^1(X) \to H^1(W) \xrightarrow{i^* - j^*} H^1(Y) \to \cdots$$ implies by rank considerations that $i^* - j^*$ has rank zero, hence (being a map between free abelian groups) $i^* = j^*$, and in particular the map $p^*: H^1(X) \to H^1(W)$ is surjective. It now follows from the equation $\delta_3 = i^* p^*$ that the index of $\text{im}(\delta_3)$ coincides with $\det(i^*) = \text{disc}(W)$.
\end{proof}

We will now prove Theorem \ref{cor:det-squared}. This is the first place we use that $X$ is an \textit{integer} homology torus; one can extend the result to rational homology tori, with $\det(X)$ replaced by $\det(X) |H_1(X)|$. 

\begin{proof}[Proof of Theorem \ref{cor:det-squared}]
If $\det(X) = 0$, then by Lemma \ref{lemma:det2} we have either $b_1(W) \ne 3$ or $b_1(W) = 3$ but $\text{disc}(W) = 0$. In either case, Theorem \ref{thm:torus-rigidity} gives $\mathfrak m(X) = 0$, as desired.
\par
More interestingly, if $\det(X) \ne 0$, then by Lemma \ref{lemma:det2} we have $b_1(W) = 3$, $b^+(W) = 0$, and $\text{disc}(W) = \det(X)$. What remains is to show that $$D(W) = \begin{cases} |H^2(Y)/f^*H^2(X)| & \text{if $W$ supports a torsion spin}^c \text{ structure} \\ 0 & \text{otherwise.} \end{cases}$$ 
By rank considerations and the fact that $H_1(Y)$ is torsion-free, in the long exact sequence $$\cdots \to H_2(Y) \to H_2(X) \to H_2(W, \partial W) \to H_1(Y)$$ the last map is zero. Thus we have an isomorphism $$H^2(W) \cong H_2(W, \partial W) \cong H_2(X)/f_* H_2(Y).$$ This group is of the form $\Bbb Z^3 \oplus T$ for a torsion abelian group $T$; using that $H_i(X)$ is torsion-free and appealing to Smith normal form, one finds that $H^2(Y)/f^* H^2(X) \cong T$, so $|H^2(Y)/f^* H^2(X)|$ coincides with the number of elements in $\text{Tors} H^2(W)$; and so long as one exists, this set is in bijection with the set of torsion spin$^c$ structures on $W$. Thus Theorem \ref{cor:det-squared} is reduced to showing that $D(W)$ counts the number of torsion spin$^c$ structures on $W$.

$D(W)$ counts spin$^c$ structures with $i^* \spin_W \cong j^* \spin_W$ torsion and $d(\spin_W) = 0$. Because $T^3$ carries a unique torsion spin$^c$ structure, the first condition is equivalent to $\spin_W$ being torsion on the ends (so that in particular $d(\spin_W)$ is defined). Because $X$ is a homology torus we have $\chi(W) = \sigma(W) = 0$ and thus $d(\spin_W) = \frac 14 c_1(\spin_W)^2$. Because $W$ is negative-definite and $$[c_1(\spin_W)] \in \hat H(W) = \text{im}(H^2(W, \partial W; \Bbb R) \to H^2(W; \Bbb R)),$$ we see that $d(\spin_W) = 0$ is equivalent to $[c_1(\spin_W)] = 0$ in real cohomology, which is equivalent to $\spin_W$ being torsion. Thus $D(W)$ counts torsion spin$^c$ structures.
\end{proof}

\subsection{Some examples}
Write $t_i(W) = |\text{Tors} H_i(W)|$ and similarly for variations, so that our formula can be rewritten $\mathfrak m(X) = \pm \det(X) t^2(W)$. It follows from the universal coefficient theorem and Poincar\'e duality that there are essentially two torsion coefficients associated to $W$:
\begin{align*}
t_1(W) = t^2(W) = t_2(W, \partial W) &= t^3(W, \partial W),\\ t_1(W, \partial W) = t^2(W, \partial W) &= t_2(W) = t^3(W).
\end{align*}
Lemma \ref{lemma:det1} computes that $t_2(W) = \det(X),$ when $X$ is an integer homology torus (and more generally for rational homology tori, $t_2(W) = \det(X) t_1(X)$). It is interesting that $t^2(W) = t_1(W)$ does not seem to have such a simple description. In this section we will discuss two concrete examples.

\begin{remark}
Combining our result with Ruberman--Strle's, we see that if $\det(X)$ is odd, then $t^2(W)$ is odd. It is not hard to show this by purely algebraic arguments; more generally, if $\det(X)$ is nonzero mod $p$, then $t^2(W)$ is also nonzero mod $p$.
\end{remark}

\begin{example}\label{ex:3mfd}
Suppose $M$ is a rational homology 3-torus containing a non-separating two-torus $T \hookrightarrow M$, and let $X = S^1 \times M$; notice that $X$ is spin and $\det(X) = \det(M)$. Denoting by $C$ the complement, we have $W \cong S^1 \times C$. Now by the K\"unneth theorem $t_1(W) = t_1(C) = t_2(W)$, and as discussed above $t_2(W) = \det(X) t_1(X) = \det(M) t_1(M)$. Therefore $$\pm \mathfrak m(X) = \det(X) t^2(W) = \det(M)^2 t_1(M).$$
When $M$ is an integer homology torus, this is consistent with the result of Meng-Taubes \cite[Section 5]{RubermanStrle}.
\end{example}

\begin{example}\label{ex:misdet}
Fix non-zero integers $n_1,n_2, n_3$. Consider the cobordism $Z$ obtained from $I\times T^3$ first by adding three $1$-handles, so that the resulting boundary is $T^3\hash(S^1\times S^2)^3$, and then three $2$-handles along curves with homology class
\begin{equation*}
(1,0,0,n_1,0,0),\quad (0,1,0,0,n_2,0),\quad (0,0,1,0,0,n_3).
\end{equation*}
By choosing the framings appropriately, we can arrange that $Z$ is spin. Call $Y$ the other boundary component. Cellular homology computations give that $H_1(Z)=\mathbb{Z}^3$, that the image of
\begin{equation*}
H_1(T^3)\rightarrow H_1(Z) 
\end{equation*}
is a sublattice of index $n_1n_2n_3$, and that
\begin{align*}
H_1(Z,T^3)&=\mathbb{Z}/n_1\oplus\mathbb{Z}/n_2\oplus\mathbb{Z}/n_3\\
H_2(Z,T^3)&=H_3(Z,T^3)=0.
\end{align*}
By Poincar\'e--Lefschetz duality
\begin{equation*}
H^1(Z,Y)=H^2(Z,Y)=0
\end{equation*}
so that
\begin{equation*}
H^1(Z)\rightarrow H^1(Y)
\end{equation*}
is an isomophism. We now take $W=Z\cup_Y \tilde{Z}$, the double of $Z$ along $Y$. The isomorphism just mentioned implies that the map in the Mayer-Vietoris sequence
\begin{equation*}
H^2(W)\rightarrow H^2(Z)\oplus H^2(\tilde{Z})
\end{equation*}
is injective, hence $H^2(W)$ is torsion-free.
\par
To sum up, $W$ is a cobordism from $T^3$ to itself such that the two inclusion map are the same map in $H_1$; gluing together the boundary components we obtain an integral homology torus which has determinant $n_1n_2n_3$ by Lemma \ref{lemma:det1}, but $H^2(W)$ is torsion-free. Therefore, for this family of spin integer homology tori $X$, we have $\pm \mathfrak m(X) = \det(X).$
\end{example}

\vspace{0.3cm}

\textbf{Conflicts of interest:} none.

\vspace{0.3cm}

\textbf{Financial support:} The first author was partially supported by NSF grant DMS-2203498.

\vspace{0.3cm}

\bibliography{biblio.bib}
\bibliographystyle{alpha}
\end{document}